\documentclass[reqno]{amsart}
\usepackage[dvipsnames]{xcolor}
\usepackage{amssymb,latexsym,amsfonts,amsmath,mathtools}
\usepackage{mathrsfs}
\usepackage{graphicx}
\usepackage{tikz}

\usepackage[normalem]{ulem}
\usepackage{comment}
\usetikzlibrary{calc}
\usepackage{pgfplots}
\pgfplotsset{compat=newest}
\usepackage[nomessages]{fp}
\usepackage{dsfont}
\usetikzlibrary{calc,shapes,arrows}
\usepackage{enumitem}

\newtheorem{theorem}{Theorem}[section]
\newtheorem{lemma}[theorem]{Lemma}

\newtheorem{prop}[theorem]{Proposition}
\newtheorem{cor}[theorem]{Corollary}

\newtheorem{defn}[theorem]{Definition}

\newtheorem{assumption}[theorem]{Assumption}
\theoremstyle{remark}
\newtheorem{remark}[theorem]{Remark}

\numberwithin{equation}{section}

\newcommand{\R}{\mathbb{R}}{}
\newcommand{\N}{\mathbb{N}}

\newcommand{\bT}{\mathbb{T}}

\newcommand{\cA}{\mathcal{A}}
\newcommand{\cB}{\mathcal{B}}

\newcommand{\cV}{\mathcal{V}}

\newcommand{\cP}{\mathcal{P}}

\newcommand{\hi}{i}
\newcommand{\hs}{s}

\newcommand{\hv}{v}
\newcommand{\hx}{x}
\newcommand{\hp}{p}

\newcommand{\tsw}{t_\star}

\newcommand{\wtt}{\widetilde{t}\,}

\newcommand{\eps}{\varepsilon}

\DeclarePairedDelimiterX{\inp}[2]{\langle}{\rangle}{#1, #2}
\newcommand{\wto}{\rightharpoonup}
\newcommand{\weaks}{\stackrel{*}{\wto}}
\newcommand{\stkout}[1]{\ifmmode\text{\sout{\ensuremath{#1}}}\else\sout{#1}\fi}

\newcommand\restr[2]{{% we make the whole thing an ordinary symbol
  \left.\kern-\nulldelimiterspace % automatically resize the bar with \right
  #1 % the function
  \vphantom{|} % pretend it's a little taller at normal size
  \right|_{#2} % this is the delimiter
  }}
\title[Optimality of vaccination for a state constrained SIR epidemic]{Optimality of vaccination for an SIR epidemic\\ with an ICU constraint}

\author{Matteo Della Rossa}
\author{Lorenzo Freddi}
\author{Dan Goreac}
\address[Matteo Della Rossa, Lorenzo Freddi]{Dipartimento di Scienze Matematiche, Informatiche e Fisiche,  Universit\`a di Udine, via delle Scienze 206, 33100 Udine, Italy}
\email{matteo.dellarossa@uniud.it, lorenzo.freddi@uniud.it}
\address[Dan Goreac]{\'{E}cole d'Actuariat, Universit\'{e} Laval, Qu\'{e}bec (Qu\'{e}bec), G1V 0A6, Canada}
\address[Dan Goreac]{School of Mathematics and Statistics, Shandong University, Weihai, Weihai 264209, PR China}
\address[Dan Goreac]{LAMA, Univ Gustave Eiffel, UPEM, Univ Paris Est Creteil, CNRS, F-77447 Marne-la-Valle\'ee, France}
 \email{dan.goreac@univ-eiffel.fr}

\begin{document}

\begin{abstract}
This paper studies an optimal control problem for a class of SIR epidemic models, in scenarios in which the infected population is constrained to be lower than a critical threshold imposed by the ICU (intensive care unit) capacity. The vaccination effort possibly imposed by the health-care deciders is classically modeled by a  control input affecting the epidemic dynamic.
After a  preliminary viability analysis  the existence of  optimal controls is established, and their structure is characterized by using a state-constrained version of Pontryagin's theorem. The resulting optimal controls necessarily have a bang-bang regime with at most one switch. More precisely, the optimal strategies impose the maximum-allowed vaccination effort in an initial period of time, which can cease only once the ICU constraint can be satisfied without further vaccination. 
The switching times are characterized in order to identify conditions under which vaccination should be implemented or halted.  The uniqueness of the optimal control is also  discussed. Numerical examples illustrate our theoretical results and the corresponding optimal strategies. The analysis is eventually extended to the infinite horizon by $\Gamma$-convergence arguments.
\end{abstract}

\maketitle

%\vspace{-6ex}

%\begin{center}
%Communicated by Irena Lasiecka.
%\end{center}



\section{Introduction}

In the face of infectious disease outbreaks, vaccination strategies play a pivotal role in curbing transmission and mitigating the impact on public health. The optimization of vaccination deployment, particularly in the context of epidemic modeling, is a critical endeavor to ensure the efficient allocation of limited resources while maximizing population immunity.

The optimal control of epidemics in the framework 
of the Susceptible-Infectious-Recovered (SIR) model of Kermack and McKendrick (\cite{kermack1927contribution})
has been extensively studied (\cite{anderson1992infectious,Behncke,hansen2011optimal,Mart,LS2011,GK2014,LT2015,ZYHetal2017,SFT2022}) in relation to both the use of pharmaceutical (e.g., optimality of vaccination) and non-pharmaceutical (e.g., slowdown 
 and lockdown) strategies.  

In this study, we delve into the optimization of vaccination strategies for an SIR epidemic, taking into account the constraints imposed by intensive care units  (ICU) capacity. 
Optimal control problems with such kind of constraints have been posed in a natural way during the recent Covid-19 pandemic  and have been investigated in the case of a non-pharmaceutical control (\cite{MSW2022,AvrFre22,AFG22cor,FGLX2022,FreddiGoreac23,BredaDel24}).  
As well as in the case of non-pharmaceutical interventions, the interplay between vaccination coverage, disease dynamics, and healthcare resource utilization presents a complex landscape that requires careful analysis and optimization. The main, well-known, mathematical difficulty lies on the fact that the ICU constraint involves the state of the system. On the other hand, the theory of state constrained optimal control problems has been settled, from the point of view of optimality conditions, only in recent years  (\cite{VP1982,HarSetSur95,BdlV2010,BdlVD2013,FMM2018}). Among the different approaches to derive optimality conditions of the first order by allowing state constraints, we choose here the one that consists in introducing multipliers that are bounded  measures, leading to adjoint (co-state) variables in the space of functions with  bounded variation (\cite{BdlV2010,BdlVD2013}).       

By exploring the optimality of vaccination deployment in scenarios with ICU constraints, this research aims to provide a fundamental, rigorous and detailed mathematical analysis capable 
to contribute to the development of
prevalence-based strategies for epidemic control and public health decision-making. 
Being concentrated on mathematical analysis, we consider here the following simple two-dimensional SIR model
\begin{equation}\label{eq:SIR2}
\begin{cases}
\dfrac{ds}{dt}(t)=-\beta\, s(t)\, i(t)-v(t)s(t),%
\\[2ex]
\dfrac{di}{dt}(t)= \beta\, s(t)\, i(t)-\gamma i(t),\\[1ex]
s(0)=s_0,\ i(0)=i_0,
\end{cases}
\quad t\in[0,T),
\end{equation}
on a time horizon with final time $T>0$ (not necessarily finite).  
Here, the states $s$ and $i$ are the population densities of susceptible and infectious individuals, respectively, while 
 the control $v(t)$ represents the vaccination rate and belongs to a bounded set $[0,v_M]$, and the derivatives are meant in a distributional sense. The constant coefficients $\beta\in(0,1]$ and $\gamma>0$ 
are, respectively, the transmission and recovery rates. 
The qualitative properties of solutions to~\eqref{eq:SIR2} are studied and illustrated in Section~\ref{sec:QualitativeAnalysis}
.

Our aim is to study 
the optimal control problem for system \eqref{eq:SIR2} which consists in minimizing, over all controls $v$ and the corresponding epidemic trajectories $s$ and $i$, the objective functional
$$
J(v,s,i)=\int_0^{T} \lambda_v v(t)+\lambda_i i(t) \,dt,
$$
% \begin{equation}\label{eq:OptimalControlProblemNoDelay}
% \begin{aligned}
% %\min_{v\in\cV} 
% J(v,s,i)&=\int_0^{T}\lambda_i i(t)+ \lambda_v v(t) \,dt\;\\
% &(s(t),i(t))'=f(s(t),i(t),v(t))\;\;\;\forall\;t\geq 0,\\
% &i(t)-i_M\leq 0,\;\;\;\forall\;t\geq 0,\\
% &(s(0),i(0))=x_0\in T,\\
% &v(t)\in V:=[0, v_M].
% \end{aligned}
% \end{equation}
under the state constraint
$$
i(t)\le i_M\quad\forall\,t\in[0,T],
$$
where $i_M>0$ represents a safety threshold for the ICUs capacity. The constant coefficients $\lambda_v\ge0$ and  $\lambda_i\ge0$ modulate the economic and health-related costs of infection with respect to the cost of vaccination. 

Generally speaking, the existence of a control able to keep the state $i(t)$ under the level $i_M$ is not ensured for all initial epidemic states $(s_0,i_0)$.  Thus, a preliminary viability analysis is performed in Section~\ref{subsec:Viability}, to characterize the maximal viable set of initial conditions.  The same analysis also provides the maximal set of, so called,  ``safe" initial states such that all controls are viable.  
In the same section we prove existence of an optimal control  whose structure will be characterized, after, by using necessary optimality condition of the first order provided by a state constrained version of Pontryagin's theorem (\cite{BdlV2010,BdlVD2013}). We prove that singular arcs cannot occur and the optimal controls are bang-bang 
with at most one switch. More specifically, the optimal vaccination strategy consists in exerting the maximum effort, $v_M$,  from the onset of the epidemic and do nothing else once this maximal  intervention has been implemented.
These results recover those obtained in \cite[Theorem 2.1]{Behncke}
where the ICU constraint is not imposed, and are in agreement with those in \cite{CFFGL_2024}. Moreover, it can be noted that they mark a  qualitative difference from the non-pharmaceutical optimal  control obtained in 
\cite{AvrFre22,AFG22cor,FreddiGoreac23}.  Indeed, in the latter, the optimal strategy consists in delaying the non-pharmaceutical control action by doing nothing in the first part of the epidemic,  by applying   the maximal control effort, only when constrained,  
until  reaching the level $i_M$ and, after, by stabilizing the infections at this level until herd immunity is reached. 

{\color{black} The simplest, non-trivial, scenario occurs when the density $i$ of the infected population is  constrained under the level $i_M$ but does not contribute to the cost functional, that is $\lambda_i=0$. In this case the vaccination must cease  exactly when the control~$0$ allows to take the infections under $i_M$. When, instead, the cost functional directly depends also on $i$, }the switching time (i.e., the instant at which the optimal  strategies switch from $v_M$ to $0$) depends on the ratio between  $\lambda_v$ and $\lambda_i$.  
{\color{black} This dependence is analyzed in Subsection \ref{subsec_charOC}.}

Subsection \ref{ss_uniqueness} is devoted to  discuss the uniqueness problem in the general case. 
Our results are  graphically illustrated at the end of Section \ref{sec:OptimalityPontryagin} with the aid of numerical examples.

In Section~\ref{sec:InfiniteHorizon} we extend our analysis to the infinite-horizon case by using a $\Gamma$-convergence argument. The main result (Theorem \ref{thm:CarInfiniteHor}) states that there exists at least an optimal bang-bang control. Moreover, in the case $\lambda_i=0$, the latter is the unique solution to the optimal control problem.

% The  optimal control of epidemics (see, for instance~\cite{anderson1992infectious,Behncke,hansen2011optimal,Mart}) starting from the classical SIR model of Kermack and McKendrick (\cite{kermack1927contribution}), has been widely studied in recent years, and in particular after the start of  COVID-19 pandemic, see~\cite{alvarez2020simple,Kruse,Ketch}.  

\textbf{Notation}. Along the whole paper the symbol $\equiv$ means ``equal almost everywhere". Moreover,  $\R_+:=[0,\infty)$. The notation used for spaces of functions is standard. Namely, for maps defined on an interval $I$ and values in a subset $E$ of $\R^d$ we denote by $C^k(I;E)$  the space of functions with continuous $k$-th derivatives,   $L^p(I;E)$ the Lebesgue space of (equivalence classes of) $p$-summable (if $p\in[1,+\infty)$) or essentially bounded (if $p=\infty$) functions,  
and $W^{1,p}(I;E)$ the Sobolev space of
(equivalence classes of) functions that are in $L^p$  together with their distributional derivative.
The reference to the set $E$ is often omitted when $E=\R$.

% \textcolor{blue}{Given an interval $I\subseteq \R$, $k\in \N\cup\{\infty\}$, and $p\geq 1$,  $W^{k,p}(I)$ denotes the $k-p$ \emph{Sobolev space} on the interval $I$.} 

% Given an interval $I\subseteq \R$, $g:I\to \R$,  and $k\in \R$, $g\equiv k$ stands for $g(t)=k$ for almost all $t\in I$. 
% Given a function $g$ of real variable and a constant $k$, the g\equiv k$ stands for $g(t)=k$ for almost all $t$ belongs to the domain of $g.   

\section{Qualitative analysis of epidemic trajectories} 
\label{sec:QualitativeAnalysis}

We now collect some properties of (solutions to) system~\eqref{eq:SIR2} under arbitrary inputs~$v\in \cV:=L^\infty((0,\infty);V)$ with $V:= [0,v_M]$ and $v_M\ge0$.  
% We thus introduce the planar system
% \begin{equation}\label{eq:SystemNoDelay}
% \begin{cases}
% \dfrac{ds}{dt}(t)=-\beta\, s(t)\, i(t)-v(t)s(t)%,\quad t\in(0,T)
% \\[2ex]
% \dfrac{di}{dt}(t)= \beta\, s(t)\, i(t)-\gamma i(t).%,\quad t\in(0,T)
% \end{cases}
% \end{equation}
In this model, the control $v$ represents the vaccination rate and $v_M$ the maximal vaccination effort that can be done in the time unit. 

%As well as the recovery rate $\gamma$, the constant $v_M$ is not constrained to be smaller than $1$. A value strictly bigger than $1$ simply means that all susceptible population can be vaccinated in less than a time unit (for instance, in a half day if time is measured in days).   

We define the vector field $f:\R^2\times \R\to \R^2$ corresponding to the Cauchy problem in~\eqref{eq:SIR2},  by setting \begin{equation}\label{eq_f}
f(s,i,v):=\big(
-\beta si-vs,\,
\beta si-\gamma i\big).
\end{equation}
% {\color{black} ma $v$ qui \'e costante? o dobbiamo far dipendere $f$ anche da $t$ (i.e.\ $f:[0,\infty)\times\R^2\times\R\to\R^2$, $f(t,s,i,v)=...$)?}

\begin{theorem}\label{th_euri}
For every $s_0,i_0\in(0,1)$, every $\beta>0$, $\gamma\ge0$ and every control $v\in \cV$, there exists a unique absolutely continuous  solution $(s,i)$  to the Cauchy problem ~\eqref{eq:SIR2}  on the interval $[0,\infty)$. Moreover,
\begin{enumerate}
\item $s$ and $i$ are strictly positive;
\item $i\in C^1([0,\infty))$,  $s\in W^{1,\infty}([0,\infty))$;
%\item the triangle $T:=\{(s,i)\in \R_+^2\;\vert\;s+i\leq 1\}$ is forward invariant
\item the solution $(s,i)$ satisfies
\begin{equation}\label{eq_siidint}
s(t)+i(t)-s_0-i_0=-\frac{1}{\beta}\int_0^t v(\xi)\frac{i'(\xi)}{i(\xi)}\,d\xi+\frac{\gamma}{\beta}\log(\frac{s(t)}{s_0})
\end{equation}
for all $t\in[0,\infty)$;
\item if $v$ is (essentially) constant equal to $v_0\in [0,v_M]$, then the solution $(s,i)$ satisfies the identity 
\begin{equation}\label{eq_siid}
s+i-s_0-i_0+\frac{v_0}{\beta}\log(\frac{i}{i_0})-\frac{\gamma}{\beta}\log(\frac{s}{s_0})=0;
\end{equation}
\item $s$ is strictly decreasing;
\item $\displaystyle i_\infty:=\lim_{t\to\infty}i(t)=0$ and  $\displaystyle s_\infty:=\lim_{t\to\infty}s(t)\in[0,\frac{\gamma}{\beta})$;
\item setting $t_\frac{\gamma}{\beta}:=\inf\{t\ge0\ \vert\ s(t)\le\frac{\gamma}{\beta}\}$, we have $t_\frac{\gamma}{\beta}\le
\frac{1}{\gamma i_0}\big(s_0-\frac{\gamma}{\beta})^+.$
\end{enumerate}
\end{theorem}

\begin{proof}
Since the function $f:\R^2\times \R\to \R^2$ is locally Lipschitz, then we have local existence and uniqueness of an absolutely continuous solution (see for instance \cite[I.5]{Hale_ODE}), for any $v\in \cV$. Let $[0,c)$,  with $c>0$, be a time interval in which the unique solution $(s,i)$ exists.  

%\vspace{1ex}

Let us claim that $s(t)>0$ and $i(t)>0$ for every $t\in [0,c)$. Indeed,
considering $i$ as a coefficient, the function $s$ is the unique solution {\color{black}to} the linear Cauchy problem
$$
\begin{cases}
s'=-(\beta i+v) s, \\
s(0)=s_0,
\end{cases}
$$
which is given by
$$
s(t)=s_0 e^{-\int_0^t[\beta i(\xi)+v(\xi)]\,d\xi}\quad\forall\,t\in [0,c).
$$
Hence, $s$ is strictly positive in $[0,c)$. Similarly, one can prove that also $i$ is strictly  positive. 

By summing the two equations of the system we get
\begin{equation}\label{d(s+i)}
(s+i)'=-vs-\gamma i< 0\ \mbox{ in }[0,c),
\end{equation}
which implies that $s+i$ is decreasing in $[0,c)$. Then we have
\begin{equation}\label{s+ile}
0<s(t)+i(t)\le s_0+i_0\quad\forall\,t\in [0,c).
\end{equation}
%and hence, since we are assuming $s_0,i_0\in (0,1)$, we have  $$(s(t),i(t))\in \bT:=\{(s,i)\in[0,1]^2\ \vert\ s+i\le1\}$$ for every $t\in [0,c)$.
{\color{black} Being positive, the solutions are  bounded. Thus, they} exists for every $t\in[0,\infty)$ and all properties above hold with $c=\infty$.  This proves  assertion {\em(1)} and the first part of the statement.

Simply looking at the equations, one sees that $i$ has continuous derivatives, while
$s$ has bounded derivatives, which proves {\em(2)}.

Assertion {\em(3)} is proved by integrating \eqref{d(s+i)} 
$$
s+i-s_0-i_0=-\int_0^tv(\xi)s(\xi)d\xi-\gamma\int_0^ti(\xi)d\xi
$$
and substituting in the integrals the following expressions of $s$ and $i$
\begin{equation}\label{eq_si}
s=\frac{1}{\beta}\frac{i'}{i}+\frac{\gamma}{\beta},\quad 
i=-\frac{1}{\beta}\frac{s'}{s}-\frac{v}{\beta},
\end{equation}
as obtained by the second and the first equation of the system, respectively. Assertion {\em(4)} is a straightforward consequence of {\em(3)}, while {\em(5)} follows from {\em(1)} and the first equation of the system.

{\em(6)} By monotonicity and positivity, the limit  $s_\infty:=\lim_{t\to\infty}s(t)$ exists and belongs to $[0,1)$.   
Since $s+i$ is decreasing (see \eqref{d(s+i)}) and positive, then $\lim_{t\to\infty}\big(s(t)+i(t)\big)$ exists and is finite. 
As a consequence, also the limit
$$
i_\infty:=\lim_{t\to\infty}i(t)=\lim_{t\to\infty}\big(s(t)+i(t)-s(t)\big)=\lim_{t\to\infty}\big(s(t)+i(t)\big)-s_\infty
$$
exists and is finite and non-negative.

Assume by contradiction that $i_\infty >  0$. Then there exists $\bar{t}$ such that $i(t)>\frac{i_\infty}{2}>0$ for any $t>\bar{t}$.   
By \eqref{d(s+i)}, for such $t$ we have 
$$
(s+i)'\le -\gamma \frac{i_\infty}{2}
$$
and hence $s+i\to-\infty$, which is impossible since $s+i>0$. Then we conclude that $i_\infty=0$. 

It remains to prove that $s_\infty<\frac{\gamma}{\beta}$. Since $s$ is strictly decreasing, if $s_0\le\frac{\gamma}{\beta}$, then there is nothing to prove. 
Let us then assume that $s_0>\frac{\gamma} {\beta}$. 
Let us denote by
$$
t_{\max}:=\sup\{t\ge0\ \vert\ s(t)>\dfrac{\gamma}{\beta}\}
$$
and assume, by contradiction, that $t_{\max}=\infty$.  We have 
$$
i'(t)=\big(\beta s(t)-\gamma\big)i(t)>0\quad \forall\, t>0,
$$
hence $i$ is increasing, {\color{black} so obtaining a contradiction with $i_0>0$ and $i_\infty=0$. Then $t_{\max}<\infty$ and, by continuity, $s(t_{\max})=\frac{\gamma}{\beta}$.  Since $s$ is strictly decreasing, this implies $s_\infty<\frac{\gamma}{\beta}$.}

To prove  {\em(7)} we can assume that $s_0>\frac{\gamma}{\beta}$ (otherwise, we have $t_\frac{\gamma}{\beta}=0$ and the inequality is trivially satisfied). Moreover, by continuity of $s$, we have
\begin{equation}\label{tgsb}
t_\frac{\gamma}{\beta}=\sup\{t\ge0\ \vert\ s(t)>\dfrac{\gamma}{\beta}\}.
\end{equation}
As long as $s(t)>\frac{\gamma}{\beta}$ we have that $i$ is increasing  because,  as before, 
$i'(t)>0$ and 
 we have
$$%s(t)\ge \frac{\gamma}{\beta-u_{\max}}\ \implies\ 
s'(t)=-\beta s(t) i(t)-v(t)s(t)\le 
-\beta s(t) i(t)
<-\gamma i(t)
\le -\gamma i_0.$$
Thus, by integrating on $[0,t]$, we get
$$
s(t)-s_0<-\gamma i_0t.
$$
and, 
by using $s(t)>\frac{\gamma}{\beta}$ again,  we obtain 
$\frac{\gamma}{\beta}-s_0<-\gamma i_0t$, which implies
$$
t<
\frac{1}{\gamma i_0}\big(s_0-\frac{\gamma}{\beta})\quad \forall t\ge0\ :\ s(t)>\frac{\gamma}{\beta},
$$
and the desired inequality follows by taking the supremum and using  \eqref{tgsb}. 
\end{proof}

\begin{remark}\label{rem:Boundedness}
As a by-product of Theorem~\ref{th_euri} (assertion ~\emph{(1)})  and of equation \eqref{s+ile}  which holds for $c=\infty$, we have that the triangle $$\bT:=\{(s,i)\in \R_+^2\;\vert \;s+i\leq 1\}$$ is \emph{forward invariant}, that is, for any $(s_0,i_0)\in \bT$ and any $v\in \cV$, the solution $(s,i)$ to the Cauchy problem~\eqref{eq:SIR2} with initial conditions $s(0)=s_0$, $i(0)=i_0$ satisfies $(s(t),i(t))\in \bT$ for all $t\in \R_+$. 
\end{remark}

\begin{remark} From \eqref{eq_siid} it is easy to see that:
\begin{enumerate}[leftmargin=*]
\item if $v\equiv0$ then  $s_\infty>0$ (see also \cite[Proposition 2.4]{FreddiGoreac23});
\item if $v\equiv v_0>0$ then $s_\infty=0$. This can be also noted by considering the function $U:\bT\to \R_+$ defined by $U(s,i):=s+i$. It can be proved that $\inp{\nabla U(s,i)}{f(s,i,v_0)}\leq -\alpha U(s,i)$ for every $(s,i)\in \bT$, with $\alpha:=\min\{v_0,\gamma\}>0$. This implies that the function $U$ is a Lyapunov function, proving that the unique equilibrium $(0,0)$ is exponentially stable in $\bT$, see for instance~\cite[Theorem 4.10]{Khalil2002}.
\end{enumerate}
\end{remark}

\begin{remark}\label{rem_cd}
Looking at {\em(4)} of Theorem \ref{th_euri}, for every constant $v\in[0,v_M]$ and every $(s_0,i_0)\in(0,1)^2$ we denote by   $F_{s_0,i_0}:(0,1)^2\to\R$ the function 
 defined as
$$
F_{s_0,i_0}(s,i):=s+i-s_0-i_0+\frac{v}{\beta}\log(\frac{i}{i_0})-\frac{\gamma}{\beta}\log(\frac{s}{s_0}).
$$
By a trivial computation,  for every $(s_0,i_0),\, (s_1,i_1)\in(0,1)^2$ we have 
$$
 F_{s_0,i_0}(s_1,i_1)=0\ \implies\ F_{s_0,i_0}\equiv F_{s_1,i_1}.
$$
This means that the family of curves implicitly defined in {\em(4)} of Theorem \ref{th_euri} by  $F_{s_0,i_0}(s,i)=0$ are coinciding or disjoint as  $(s_0,i_0)$ varies in $(0,1)^2$. In other words,  the trajectories of solutions corresponding to the same constant control $v$ and starting from different initial conditions are coinciding or disjoint. 
\end{remark}

We now provide a concluding lemma, establishing the boundedness of the \emph{integral}, over the whole interval $[0,\infty)$, of the $i$-component of the solution. 
\begin{lemma}\label{lemma:BoundIntegral}
For every $(s_0,i_0)\in \bT$, every $\beta>0$, $\gamma\ge0$ and every control $v\in \cV$, let us denote by $(s,i):\R_+\to \R^2$ the solution to the Cauchy problem~\eqref{eq:SIR2} with initial conditions $s(0)=s_0$, $i(0)=i_0$, and with respect to the control $v$. The integral
\[
\int_0^{\infty} i(t)\;dt
\]
is finite.
\end{lemma}

\begin{proof}
By Item~\emph{(6)} of Theorem~\ref{th_euri} we have that $s_\infty=\lim_{t\to \infty}s(t)<\frac{\gamma}{\beta}$, let us suppose that $s_\infty=\frac{\gamma}{\beta}-2\varepsilon$, for some $\varepsilon>0$. This implies that there exists $T>0$ such that $s(t)\leq \frac{\gamma}{\beta}-\varepsilon$, for all $t\geq T$. We have
\[
i'(t)=\beta s(t)i(t)-\gamma i(t)\leq \beta (\frac{\gamma}{\beta}-\varepsilon)i(t)-\gamma i(t)=-\beta \varepsilon i(t),\;\;\;\forall \,t\geq T.
\]
By Gr\"onwall's inequality (see~\cite[Lemma 2.7]{Teschl12}), we have that $i(t)\leq i(T)e^{-\beta\varepsilon (t-T)}$ for all $t\geq T$.
Concluding, this implies that
\[
\int_0^{\infty} i(t)\,dt=\int_0^{T} i(t)\,dt+\int_T^{\infty} i(t)\,dt\leq \int_0^{T} i(t)\,dt+i(T)\int_T^{\infty} e^{-\beta\varepsilon (t-T)}\,dt <\infty,
\]
as required.
\end{proof}

\section{Viability analysis and optimal control problem}\label{sec:MAInSection}
We now want to introduce and study a class of optimal control problems, in which the state dynamics are driven by the SIR model~\eqref{eq:SIR2}. 
We introduce the following notational convention.
\begin{remark}[Notation]\label{dsn}{
 Whenever the control $v\in \cV$ and the initial conditions $s(0)=s_0,\ i(0)=i_0$ are fixed,  the unique solution to \eqref{eq:SIR2} will be also denoted by $x^{s_0,i_0,v}:=(s^{s_0,i_0,v}, i^{s_0,i_0,v})$. However, in the sequel, when the initial conditions are fixed or can be easily deduced from the context, we simply write  $x^v=(s^{v}, i^{v})$, or even $x=(s, i)$ if also the control is fixed or easy to deduce. 
This notation will be used  throughout the whole paper. 
 }
 \end{remark}

% In addition to the notation just introduced, for the rest of this section, we enforce the assumptions of Theorem \ref{esT}, that is, \[i_0\in(0,1),\ s_0\in(0,1-i_0],\ u\in U.\] %Moreover $s=s^{s_0,i_0,u}$ and $i=i^{s_0,i_0,u}$, according to the notation just introduced above.

Given some $\lambda_i,\lambda_v\geq 0$, $\gamma,\beta\geq 0$, $v_M\ge0$,  $0<i_M\leq 1$ we consider  the optimal control problem
\begin{equation}\label{eq:OptimalControlProblemNoDelay}
\begin{aligned}
\mbox{minimize }%_{v\in\cV} 
J(v,s,i)&=\int_0^{T}\lambda_v v(t)+\lambda_i i(t)  \,dt,\;\\
&(s(t),i(t))'= f(s(t),i(t),v(t)),\quad%\text{a.e. in }
t\in(0,T), \\
&(s(0),i(0))=(s_0,i_0)\in \mathbb{T},\\
&i(t)\leq i_M\,\quad%\text{a.e. in }
\forall t\in(0,T),\\
&v\in L^\infty((0,T);V),\  V:=[0, v_M],
%v(t)\in V:=[0, v_M]\quad\text{a.e. in }(0,T),
\end{aligned}
\end{equation}
where $f$, defined in \eqref{eq_f}, is the vector field associated to the SIR model  \eqref{eq:SIR2}, 
and the fixed final time $T>0$ can be finite or not. In the sequel, this problem will be denoted by $\cP^{T,s_0,i_0}_{\lambda_v,\lambda_i,i_M}$. Accordingly, $\cP^{\infty,s_0,i_0}_{\lambda_v,\lambda_i,i_M}$ will denote the problem on the \emph{infinite horizon} $[0,\infty)$.
Sometimes, to shorten notation, the reference to the initial conditions or to the coefficients $\lambda_v$ and $\lambda_i$  will be  dropped when they are not relevant or can be deduced by the context.

\begin{remark}
Let us observe that, by Remark \ref{rem:Boundedness}, we have that $i(t)\le1$ for every $t\ge0$. Hence, the case $i_M=1$ corresponds to a state-unconstrained case (i.e., problem  $\cP^T_{\lambda_v,\lambda_i,1}$ coincides with \eqref{eq:OptimalControlProblemNoDelay} without the state constraint $i\le i_M$). This unconstrained case has been studied in \cite[Section 2]{Behncke}. The results of that paper concerning the optimal  vaccination policy are recovered by us, in this more general setting,  except for the   uniqueness of the optimal control conjectured in Remark 2 of  \cite{Behncke}. On the other hand, its proof seems not to be completely achieved yet.  The subsequent  Subsection \ref{ss_uniqueness} is devoted to discuss this topic.  
\end{remark}

Occasionally, in the sequel, we will make use of the  {\color{black} following proposition,  usually called}  {\em principle of optimality}.

\begin{prop}[principle of optimality]\label{Prop:PrincipleOfOptim}
If $v$ is an optimal control for problem $\cP^{T,s_0,i_0}_{\lambda_v,\lambda_i,i_M}$, then for every $\eps\in(0,T)$ the control $v_\eps(\cdot):=v(\, \cdot\, +\,\eps)$ is optimal for    problem $\cP^{T-\eps,s_\eps,i_\eps}_{\lambda_v,\lambda_i,i_M}$ with the initial condition $s_\eps=s^{s_0,i_0,v}(\eps)$ and $i_\eps=i^{s_0,i_0,v}(\eps)$.   
\end{prop}
\begin{proof}
{\color{black}
By contradiction, let us suppose that there exists $w:[0,T-\eps]\to\R$  such that
$$
\int_0^{T-\eps}\lambda_v w(t)+\lambda_ii^{s_\varepsilon, i_\varepsilon,w}(t)\,dt<
\int_0^{T-\eps}\lambda_v v_\varepsilon(t)+\lambda_ii^{s_\varepsilon, i_\varepsilon,v_\varepsilon}(t)\,dt.
$$
Define  $\widetilde v:[0,T]\to \R$ by
\[
\widetilde v(t)=\begin{cases}
v(t)&\text{if }t\in [0,\varepsilon),\\
w(t-\varepsilon) &\text{if }t\in [\varepsilon, T].
\end{cases}
\]
We have that $\widetilde v$ is an admissible control and 
\[
\begin{aligned}
\int_0^T \lambda_v \widetilde v(t)+\lambda_ii^{s_0, i_0,\widetilde v}(t)\,dt&=\int_0^\varepsilon \lambda_v  v(t)+\lambda_ii^{s_0, i_0, v}(t)\,dt+\int_\varepsilon^{T}\lambda_v \widetilde v(t)+\lambda_ii^{s_\varepsilon, i_\varepsilon,\widetilde v}(t)\,dt\\
&=\int_0^\varepsilon \lambda_v  v(t)+\lambda_ii^{s_0, i_0, v}(t)\,dt+\int_0^{T-\eps}\lambda_v w(t)+\lambda_ii^{s_\varepsilon, i_\varepsilon,w}(t)\,dt<
\\&<\int_0^\varepsilon \lambda_v  v(t)+\lambda_ii^{s_0, i_0, v}(t)\,dt+ \int_0^{T-\eps}\lambda_v v_\varepsilon(t)+\lambda_ii^{s_\varepsilon, i_\varepsilon,v_\varepsilon}(t)\,dt\\&
=\int_0^T \lambda_v  v(t)+\lambda_ii^{s_0, i_0, v}(t)\,dt
\end{aligned}
\]
which contradicts the optimality of $v$.}
\end{proof}
%The validity of the principle can be easily checked, for the problem under consideration, by a contradiction argument.

\subsection{Viability Analysis}\label{subsec:Viability}
% In what follows, we will separately consider these two cases (finite horizon and infinite horizon), providing connections between the corresponding solutions. In what follows, by $\cV$, we denote the set $L^\infty([0,T], V)$ if $T$ is finite and $L^\infty([0,\infty), V)$ in the case $T=\infty$. (TO DECIDE IF EXPLICITLY PUT $T$ IN THE NOTATION 
% }
This subsection is devoted to identify for which initial conditions the state-constraint $i(t)\leq i_M$ is satisfied/satisfiable by the solutions to the state equation of problem~\eqref{eq:OptimalControlProblemNoDelay}. 

Let us introduce the state-constraint set  
\begin{equation}\label{eq:StateConstSetC}
C:=\{(s,i)\in \bT\;\vert\;i\leq i_M\}.
\end{equation}
The notions of forward invariant and viable subregions of $C$ are then introduced according to the next definition.

\begin{defn}\label{def_viableset}
The set 
$$
\cB:=\{(s_0,i_0)\in C\ \vert\ \exists\,v\in \mathcal{V}\mbox{ s.t.\ }i^{s_0,i_0,v}(t)\leq i_M\ \forall\,t\ge0\}
$$
is called  {\em maximal  feasible} or {\em viable} set. 

The set 
$$
\cA_0:=\{(s_0,i_0)\in C\ \vert \ i^{s_0,i_0,0}(t)\leq i_M\ \forall\,t\ge0\}
$$
is called {\em no-effort} or {\em safe} set. 

The set 
$$
\cA:=\{(s_0,i_0)\in C\ \vert\ i^{s_0,i_0,v}(t)\leq i_M\ \forall \,v\in \cV,\ \forall\,t\ge0\}
$$
is called \emph{maximal forward invariant} set.  
\end{defn}

\noindent The viable set $\cB$ is the maximal set of initial configurations on which at least one trajectory satisfies the constraint. When the epidemic initial state $(s_0,i_0)$ is inside $\cB$, the epidemic is ``under control". The safe zone $\cA_0$ is the maximal set of initial configurations such that the associated uncontrolled trajectories (i.e. considering $v\equiv 0$) satisfy the constraint.
The set $\cA$ is the maximal set of initial conditions in $C$ for which \emph{any} trajectory starting in $\cA$ stays in $C$ (and thus in $\cA$). 
These sets play an important role in the synthesis of optimal control strategies.

\begin{theorem}\label{lemma:ViabilityDelayFree}
The following propositions hold true.
\begin{enumerate}[leftmargin=*]
    \item
The maximal forward invariant set coincides with the no-effort set, i.e., $\cA=\cA_0$,  and it coincides with
\[
A:=\{(s,i)\in C\;\vert\;i\leq \Gamma_A(s)\},
\]
where 
\begin{equation}\label{eq:GammaBar}
\Gamma_A(s):=\begin{cases}
i_M,\;\;\;\;&\text{if }0\leq s\leq \frac{\gamma}{\beta},\\
\frac{\gamma}{\beta}+i_M-s+\frac{\gamma}{\beta}\log(\frac{\beta}{\gamma}s),\;\;\;&\text{if }s> \frac{\gamma}{\beta}.
\end{cases}
\end{equation}
\item The maximal viable set $\cB$ coincides with
\[
 B:=\{(s,i)\in C\;\vert\;i\leq \Gamma_B(s)\},
\]
for the function $\Gamma_B:[0,\infty)\to [0,i_M]$ given by
$$
\Gamma_B(s):=\begin{cases}
i_M&\mbox{if }0\le s\le\frac{\gamma}{\beta},\\
\varphi(s)&\mbox{if }s>\frac{\gamma}{\beta},
\end{cases}
$$
where $\varphi:(\frac{\gamma}{\beta},\infty)\to [0,i_M]$ is
implicitly defined by equation \eqref{eq_siid}  with $v=v_M$ and $(s_0,i_0)=(\frac{\gamma}{\beta},i_M)$, as $i=\varphi(s)$; in other words, for every $s>\frac{\gamma}{\beta}$, $i=\varphi(s)$ is the unique solution of the equation (in the unknown~$i$) 
\begin{equation}\label{eq_phi}
F(s,i):=s+i-\frac{\gamma}{\beta}-i_M+\frac{v_M}{\beta}\log(\frac{i}{i_M})-\frac{\gamma}{\beta}\log(\frac{\beta s}{\gamma})=0.
\end{equation}
\end{enumerate}
\end{theorem}

\begin{figure}[h!]\label{et0max_fig}
\begin{center}
\includegraphics[width=0.95\textwidth]{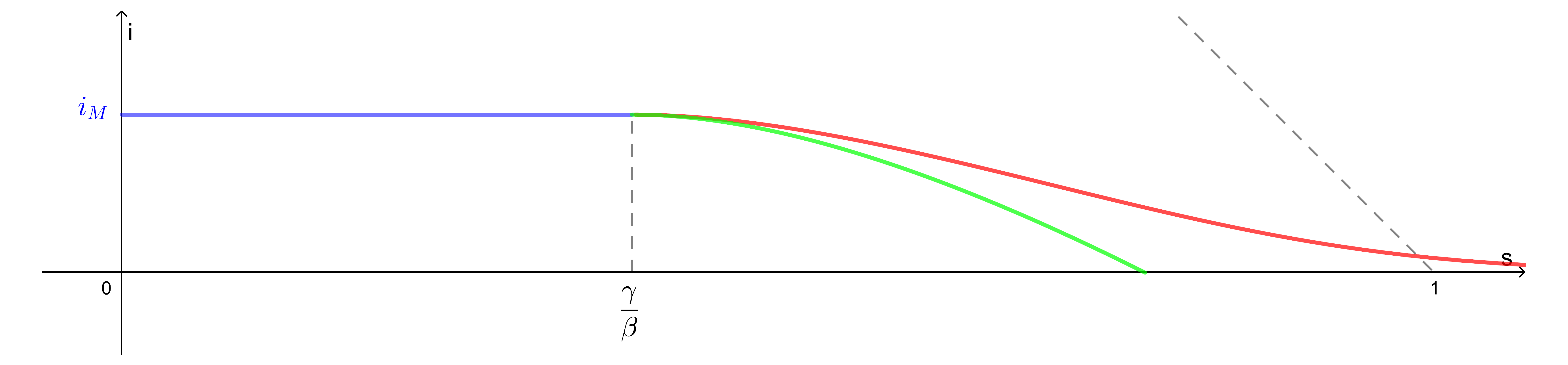}%{vaccination_viability_curves_invcol1_cont.pdf}
\end{center}
\caption{The curves $\Gamma_A$ (in green) and $\Gamma_B$ (in red) for $s>\frac{\gamma}{\beta}$}
\end{figure}

\begin{remark}
The functions $\Gamma_A$ and $\Gamma_B$ are continuously differentiable (see the proof below). The set $\cA$ is convex, while $\cB$ is, in general, non-convex. 
\end{remark}

% \begin{theorem}\label{lemma:ViabilityDelayFree}
% The following propositions hold true:
% \begin{enumerate}[leftmargin=*]
%     \item
% The maximal forward invariant set contained in $C$ is given by
% \[
%  A:=\{x=(s,i)\in C\;\vert\;s\leq \Gamma_A(s)\},
% \]
% where 
% \begin{equation}\label{eq:GammaBar}
% \Gamma_A(s)=\begin{cases}
% i_M,\;\;\;\;&\text{if }0\leq s\leq \frac{\gamma}{\beta},\\
% \frac{\gamma}{\beta}+i_M-s+\frac{\gamma}{\beta}\log(\frac{\beta}{\gamma}s),\;\;\;&\text{if }s\geq \frac{\gamma}{\beta}.
% \end{cases}
% \end{equation}
% In particular, the set $A$ is bounded and convex.
% \item The maximal viable set contained in $C$, is given by
% \[
%  B:=\{x=(s,i)\in C\;\vert\;s\leq \Gamma_B(s)\},
% \]
% for a curve $\Gamma_B:[0,\infty)\to [0,i_M]$ such that 
% \[
% \begin{aligned}
% \Gamma_B(s)&=
% i_M,\;\;\;\text{if 
%  }\ 0\leq s\leq \frac{\gamma}{\beta},\\
% \lim_{s\to \infty} \Gamma_B(s)&=0.
% \end{aligned}
% \]
%     \end{enumerate}
%     The set $B$ is in general unbounded and non-convex.
% \end{theorem}

\begin{proof}
{\em(1)} {\color{black} The no-effort set $\cA_0$ is obtained by considering the uncontrolled SIR system. It has been  
characterized in~Theorem 2.3 of~\cite{AvrFre22} and coincides with $A$. It remains to prove that also  $\cA=A$.

We start by observing that, by definition, the inclusion $\cA\subseteq\cA_0=A$ holds true.

We prove the remaining opposite inclusion,
$A\subseteq \mathcal{A}$,
by showing that when the initial point is taken on the boundary $\partial A$, then any control $v\in[0,v_M]$ keeps the trajectory  inside the set $A$. 
The cases in which $(s_0,i_0)\in \partial A$ is such that $s_0=0$,  $i_0=0$ or $i_0=i_M$ are straightforward.
Therefore, we turn our attention to  the less trivial case in which
$(s_0,i_0)\in\partial A$ is such that $s_0>\frac{\gamma}{\beta}$ and $i_0=\Gamma_A(s_0)$. The exterior normal vector to $\partial A$ in $(s_0,i_0)$ is orientated as $(\nu_1,\nu_2)=(-\Gamma_A'(s_0),1)=(1-\frac{\gamma}{\beta s_0},1)$ and it can be verified that
$$
\inp{(\nu_1,\nu_2)}{f(s_0,i_0,v)}\le0\;\;\;\;\forall \,v\in [0,v_M]
%=\varphi'(s_0)(\beta i+u)s+(\beta s-\gamma) i
% =-\frac{\beta s-\gamma}{s}\frac{i}{\beta i+v_M}(\beta i+u)s+(\beta s-\gamma) i
% =(\beta s-\gamma)i \big[1-\frac{\beta i+v_M}{\beta i+v_M}\big]\le0
$$
which means that any control $v\in[0,v_M]$ keeps the trajectory inside the set $A$.}

% It remains to prove the opposite inclusion $\cA\subseteq A$. 

% Let us then consider $(s_0,i_0)\in C$ such  that $i^{s_0,i_0,v}(t)\leq i_M$ for all $t\in [0,\infty)$ and all $v\in \mathcal{V}$.

% We have only to consider the case in which 
% $s_0>\frac{\gamma}{\beta}$ since, otherwise,  $(s_0,i_0)\in\{(s,i)\in C\ :\ 0<i\le\Gamma_A(s)\}=A$ and the inclusion is proved.
% We have to prove that $i_0\le \Gamma_A(s_0)$, that is 
% \begin{equation}\label{eq_cAsA}
% i_0\le \frac{\gamma}{\beta}+i_M-s_0+\frac{\gamma}{\beta}\log(\frac{\beta}{\gamma}s_0). 
% \end{equation}
% By taking $v=0$ and applying  {\em(4)} of Theorem \ref{th_euri} with $v_0=0$,  $i=i^{s_0,i_0,0}$, $s=s^{s_0,i_0,0}$, and since $i\le i_M$, we get
% \begin{equation*}
% \begin{split}
% i_0&\le s(t)+i_M-s_0-\frac{\gamma}{\beta}\log(\frac{s(t)}{s_0})\\
% % &\le s(t)+i_M-\frac{1}{\beta} v_M\log(\frac{i_0}{i_M})-\frac{\gamma}{\beta}\log(\frac{s(t)}{s_0})
% \end{split}
% \end{equation*}
% which holds for every $t\ge0$. Since, by point {\em(6)} of Theorem \ref{th_euri} and by continuity, we have that there
% exists $\bar{t}\ge0$ such that  $s(\bar{t})=\frac{\gamma}{\beta}$, then the evaluation of the previuos inequality in $t=\bar{t}$ provides the desired inequality \eqref{eq_cAsA}. 
% }

{\em(2)} The maximal viable set contained in $C$ is the set $\mathcal{B}$  of Definition \ref{def_viableset}.

First of all we note that 
$F\in C^1((0,\infty)^2)$
and
$$
\frac{\partial F}{\partial i}=1+\frac{v_M}{\beta i} > 0 \quad \forall\, i >0,
$$
hence, by Dini's Implicit Function Theorem, the equation $F(s,i)=0$ implicitly defines a unique function $\varphi:(\frac{\gamma}{\beta},\infty)\to\R$ such that
$F(s,\varphi(s))=0$ for every $s\in(\frac{\gamma}{\beta},\infty)$. Hence the function $\Gamma_B$ and the set $B$ are well defined. Moreover, 
$$
\frac{\partial F}{\partial s}=1-\frac{\gamma}{\beta s} > 0 \quad \forall\, s >\frac{\gamma}{\beta}.
$$
Then
$$
\varphi'(s)=-\frac{{\partial F}/{\partial s}}{{\partial F}/{\partial i}}%=
%-\frac{1-\frac{\gamma}{\beta s}}{1+\frac{v_M}{\beta i}}
%=-\frac{\beta s-\gamma}{\beta s}\frac{\beta i}{\beta i+v_M}
=-\frac{\beta s-\gamma}{s}\frac{i}{\beta i+v_M}
%=
%-\frac{\beta s-\gamma}{\beta i+v_M}\frac{i}{s}
<0\quad 
\forall\, s >\frac{\gamma}{\beta}.
$$
Thus, $\varphi$ is strictly decreasing, positive and such that $\sup\varphi=\varphi(\frac{\gamma}{\beta}^+)=i_M$ and $\varphi'(\frac{\gamma}{\beta}^+)=0$. In particular, $\Gamma_B\in C^1([0,\infty))$.

The viability of the set $B$, that is the inclusion
$$
B\subseteq \mathcal{B},
$$
is now proved  by choosing $v\equiv v_M$ and showing that, when the initial point is taken on the boundary $\partial B$, the trajectory keeps the epidemic state inside the set $B$. 
The cases in which $(s_0,i_0)\in \partial B$ is such that $s_0=0$,  $i_0=0$ or $i_0=i_M$ are straightforward.
Let us concentrate on the less trivial case in which
$(s_0,i_0)\in\partial B$ is such that $s_0>\frac{\gamma}{\beta}$ and $i_0=\varphi(s_0)$. The exterior normal vector to $\partial B$ in $(s_0,i_0)$ is orientated as $(\nu_1,\nu_2)=(-\varphi'(s_0),1)$ and it can be verified that
$$
\inp{(\nu_1,\nu_2)}{f(s_0,i_0,v_M)}=0
%=\varphi'(s_0)(\beta i+u)s+(\beta s-\gamma) i
% =-\frac{\beta s-\gamma}{s}\frac{i}{\beta i+v_M}(\beta i+u)s+(\beta s-\gamma) i
% =(\beta s-\gamma)i \big[1-\frac{\beta i+v_M}{\beta i+v_M}\big]\le0
$$
which means that the control $v\equiv v_M$ keeps the trajectory on the boundary and, hence, inside the set $B$.

% In the case in which $s_0=1-i_0$ and $i_0<\varphi(s_0)$ (see Figure \ref{}) 
% we have $(\nu_1,\nu_2)=(,) $

It remains then to prove the opposite inclusion $\cB\subseteq B$.

Let us then consider $(s_0,i_0)\in C$ and $v\in \mathcal{V}$ be such  that $i^{s_0,i_0,v}(t)\leq i_M$ for all $t\in [0,\infty)$.
We have only to consider the case in which 
$s_0>\frac{\gamma}{\beta}$ since, otherwise,  $(s_0,i_0)\in\{(s,i)\in C\ :\ 0<i\le\Gamma_B(s)\}=B$ and the inclusion is proved.
We aim to prove that $i_0\le \varphi(s_0)$, that is {\color{black}$F(s_0,i_0)\le0$. The equivalence between the two conditions can, indeed,  be proved by introducing the functions $f(i):=i+\frac{v_M}{\beta}\log(\frac{i}{i_M})$  and $g(s):=-s+\frac{\gamma}{\beta}+i_M+\frac{\gamma}{\beta}\log(\frac{\beta s}{\gamma})$  and observing that, by the definition \eqref{eq_phi} of $F$, the inequality $F(s_0,i_0)\le 0$ is equivalent to $f(i_0)\le g(s_0)$. On the other hand, the function $f$ is strictly increasing, hence invertible, and we obtain the equivalent proposition $i_0\le f^{-1}(g(s_0))$. The conclusion follows then by noting that $\varphi=f^{-1}\circ g$.

In fact, we are going to prove the  inequality $F(s_0,i_0)\le0$ in the following equivalent form }
\begin{equation}\label{tfm}
s_0+i_0\le \frac{\gamma}{\beta}+i_M-\frac{v_M}{\beta}\log(\frac{i_0}{i_M})+\frac{\gamma}{\beta}\log(\frac{\beta s_0}{\gamma}).
\end{equation}
By {\color{black}{\em(3)}} of Theorem \ref{th_euri}, denoting by $i=i^{s_0,i_0,v}$ and $s=s^{s_0,i_0,v}$, we have
$$
s_0+i_0=s(t)+i(t)+\frac{1}{\beta}\int_0^t v(\xi)\frac{i'(\xi)}{i(\xi)}\,d\xi-\frac{\gamma}{\beta}\log(\frac{s(t)}{s_0})
$$
for every $t\ge0$. Since $i\le i_M$, $v\le v_M$, 
and since $\frac{i'}{i}=\beta s-\gamma>0$ as long as $s>\frac{\gamma}{\beta}$, by estimating the integral we get
\begin{equation*}
\begin{split}
s_0+i_0&\le s(t)+i_M+\frac{1}{\beta} v_M\log(\frac{i_M}{i_0})-\frac{\gamma}{\beta}\log(\frac{s(t)}{s_0})\\
% &\le s(t)+i_M-\frac{1}{\beta} v_M\log(\frac{i_0}{i_M})-\frac{\gamma}{\beta}\log(\frac{s(t)}{s_0})
\end{split}
\end{equation*}
which holds for every $t\ge0$. Since, by point {\color{black}{\em(6)}} of Theorem \ref{th_euri} and by continuity, we have that there
exists $\bar{t}\ge0$ such that  $s(\bar{t})=\frac{\gamma}{\beta}$, then the evaluation of the previous inequality in $t=\bar{t}$ provides the desired inequality \eqref{tfm}. 
\end{proof}

Summarizing, in this subsection, we have proved that \emph{feasible} control policies exist if and only if the initial state lies in the set $B$. The set $A$, instead, characterizes the set of initial conditions for which the ``no-action'' choice ($v\equiv 0$) is feasible.

\subsection{Existence of solutions for a general cost on a finite horizon}\label{sec_OCP_exist}

In this subsection we prove the existence of a solution to an optimal control problem of the form
\eqref{eq:OptimalControlProblemNoDelay} 
but with a more general cost functional.
More precisely, our problem consists in minimizing a 
{\em cost functional} $J:  \cV_T\times Y\to[0,\infty]$
of the form 
\begin{equation}\label{cost}
J(v,s,i)=\int_0^{T}f_0(t,s,i,v)\,dt,
\end{equation}
on a time horizon $I=(0,T)$, for a finite $T>0$, over all controls $v$ in the {\em space of controls} $\cV_T=L^{\infty}(I;V)$ (with $V=[0,v_M]$) and the corresponding  trajectories $(s,i)$ that are solutions, in {\em space of states} $Y=W^{1,\infty}(I;\R^2)$, to 
the {\em state equations} \eqref{eq:SIR2}  with initial conditions $(s_0,i_0)$ in the viable set $\mathcal{B}=B$ and 
fulfil the {\em state constraint}  
\begin{equation}\label{icug}
i(t)\le i_M\quad \forall \, t\in I.
\end{equation}

 The integrand 
   $f_0$ is a given {\em running cost} such that the integral makes sense. 
An {\em optimal solution} to the control problem \eqref{cost}-\eqref{icug}-\eqref{eq:SIR2} is a vector function 
$(v,s,i)\in L^\infty(I;V)\times W^{1,\infty}(I;\R^2)$ that minimizes the cost $J$ and satisfies the set of state equations and the upper bound on $i$. The function $v$ is an {\em optimal control} and $(s,i)$ an {\em optimal state or trajectory}.   We consider, for the moment,  only the case of finite final time $T>0$. 
% In the subsequent Section~\ref{sec:InfiniteHorizon} we will ensure existence of optimal controls in the infinite-horizon case, with a different proof technique.
In the subsequent Section~\ref{sec:InfiniteHorizon}, the existence of optimal controls on the infinite-horizon case will be obtained by viewing the latter as a $\Gamma$-limit of the finite-horizon problems on   $[0,T]$ as $T\to\infty$.

The following existence theorem for a very general cost functional holds.

\begin{theorem}\label{gen_ex_th}
Let $(s_0,i_0)\in B$. If $f_0:I\times\R^2\times\R\to[0,\infty)$ is a normal convex integrand, that is,  
it is measurable with respect to the Lebesgue $\sigma$-algebra on $I$ and the Borel $\sigma$-algebra on $\R^2\times\R$ 
and there exists a subset $N\subset I$ of Lebesgue measure zero such that  
\begin{enumerate}
\item $f_0(t,\cdot,\cdot,\cdot)$ is lower semicontinuous for every $t\in I\setminus
N$,
\item $f_0(t,s,i,\cdot)$ is convex for every $t\in I\setminus
N$ and $s,i\in\R$,
\end{enumerate}
then 
% \begin{enumerate}
% \item[a.] the cost functional $J$ defined by \eqref{cost} is weakly* lower semicontinuous,% on the space $L^\infty(I;K)\times W^{1,\infty}(I;\R^2)$, 
% \item[b.] 
there exists an optimal solution $(v,s,i)$ to the control problem \eqref{cost}-\eqref{icug}-\eqref{eq:SIR2}.
%\end{enumerate}
\end{theorem}

\noindent To prove the existence of an optimal solution we  observe that it is equivalent to prove the existence of a minimizer
of the functional $F:L^\infty(I;V)\times W^{1,\infty}(I;\R^2)\to \R$ defined by
\begin{equation}\label{minF}
F(v,s,i):=J(v,s,i)+\chi_{\Lambda}(v,s,i)+\chi_{i\le i_M}(i)
\end{equation}
where $\Lambda\subset L^\infty(I;V)\times W^{1,\infty}(I;\R^2)$ is the set of admissible pairs, that is all control-state  vectors $(v,s,i)$ that satisfy the initial value problem \eqref{eq:SIR2}, while $\chi_\Lambda $ denotes the indicator function of $\Lambda$ that takes the value $0$ on $\Lambda$ and $\infty$ otherwise; similarly, the function $\chi_{i\le i_M}(i)$ is $0$ if $i(t)\le i_M$ for every $t\in I$, and $\infty$ otherwise. 

\vspace{1ex}
 
\begin{proof}
On the domain of $F$, that is the space $L^\infty(I;V)\times W^{1,\infty}(I;\R^2)$, we consider the topology given by the product of the weak* topologies of the two spaces and aim to prove sequential lower semicontinuity and coercivity  
of the functional $F$ with respect to this topology. By the 
Direct Method of the Calculus of Variations {\color{black}(see, for instance, Buttazzo \cite[Sec.\ 1.2]{Buttazzo89})}, these properties imply the existence of a solution to the minimum problem. They are direct consequences of the  
fact that the space of controls is weakly* compact, that the assumptions on $f_0$ imply that the cost functional $J$ is weakly* lower semicontinuous (see for instance \cite[Theorem 7.5]{FL_MMCVLp} or \cite[Section 2.3]{Buttazzo89}) 
 and the fact that the sets $\Lambda$ {\color{black} and $\{i\le i_M\}$ are} closed with respect to the weak* convergence. 
{\color{black} The claimed closedness of such sets follows by the application of Rellich compactness theorem, which ensures that weakly* converging sequences in $W^{1,\infty}(I)$ are, up to subsequences, uniformly converging on $I$.} 
\end{proof}

\begin{remark}{
The requirement on $f_0=f_0(t,s,i,v)$ to be a normal convex integrand is satisfied, in particular, if it is a piecewise continuous function of $t$, continuous in $(s,i)$ and convex in $v$. %Assumptions of this kind are satisfied in the applications.
}
\end{remark}

\begin{cor}\label{cor_ocfc} Consider $(s_0,i_0)\in B$ and any finite $T>0$. 
The optimal control problem 
\eqref{eq:OptimalControlProblemNoDelay} admits an optimal solution with a finite cost. 
\end{cor}

\begin{proof}
     The existence is a straightforward consequence of Theorem \ref{gen_ex_th} with $f_0(t,v,s,i)=\lambda_i i+\lambda_vv$.  The finiteness of the cost of the optimal solution $(v,s,i)$ (and, hence, of any optimal solution) follows by the fact that  $(s_0,i_0)\in B$  implies that $F(v,s,i)=\int_0^{T}\lambda_v v(t)+\lambda_i i(t)  \,dt$, which is finite. 
     \end{proof}

\section{Optimality via Pontryagin Principle on a finite horizon}\label{sec:OptimalityPontryagin}

The aim of this section is to characterize the optimal controls for~\eqref{eq:OptimalControlProblemNoDelay} via Pontryagin principle in the case $T<\infty$.

\subsection{Necessary optimality conditions.}

We first introduce the Hamiltonian $H:\R^2\times V\times\R^3\to \R$ by setting
\[
\begin{aligned}
H(s,i,v,p_0, p_s,p_i)&=p_0\lambda_v v+p_0\lambda_i i-p_s\beta si-p_svs+p_i\beta si-\gamma p_ii.
\end{aligned}
\]
 Given $T\in(0,\infty)$, by $BV([0,T])$ we denote 
 % the space of bounded variation functions from $[0,T]$ to $\R$. Formally, functions in $BV([0,T])$ are defined on $\R$ and required to be constant on $\R\setminus [0,T]$; we refer to~\cite[Section 2.2.]{BdlVD2013}\  for a precise definition {\color{black} bisogna darla anche per introdurre lo spazio delle misure}.
%We define $BV([0,T])$, 
the space of functions with bounded variation, defined as  
\begin{equation*}
BV([0,T]):=\{h\in L^{1}_{\rm loc}(\R)\ \vert\ h'\in M(\R),\, {\rm supp}(h')\subseteq [0,T]\},
	\end{equation*}
	where $h'$ is the distributional derivative of $h$, and $M(\R)$ stands for the set of bounded Borel measures on $\R$. Also, the space of functions with bounded variation is given by extending functions in a constant way on $\R$ (see for instance~\cite[Section 2.2.]{BdlVD2013}). More precisely, for $h \in BV([0,T])$, there exist $h_{0^-},\,h_{T^+} \in \mathbb{R}$  such that
	\begin{equation*}
		\begin{split}
			&h=h_{0^-} \quad a.e. \enspace on\enspace (-\infty,0), \\
			&h=h_{T^+} \quad a.e. \enspace on\enspace (T,\infty).
		\end{split}
	\end{equation*}
 
% We thus consider, by convention,  $h(0^-)$ and $h(T^+)$ the limiting values of a function $h\in BV([0,T])$; we refer to~\cite[Section 2.2.]{BdlVD2013} for detailed discussion. 
%Moreover, by $BV_+([0,T]$ we denote the set of functions in $BV([0,T])$ with a non-decreasing representative. Given $h\in  BV([0,T])$, by $dh$ we denote its distributional derivative, that is a bounded measure. 
It is, moreover, well known that each $BV$ function $h$ admits finite left and right limits in every point $t$ of the domain, that will be denoted by $h(t^-)$ and $h(t^+)$, respectively. {\color{black} 
This allows one to define the \emph{jump of} $h$ at any  point $t_0$ as $[h(t_0)]:=h(t_0^+)-h(t_0^-)$.} 
Moreover, there exist
unique left- and right-continuous representatives. For the sake of simplicity, in the sequel we tacitly assume to  always consider the left-continuous representative. As a consequence, for every $h\in BV$ we have $h(t)=h(t^-)$.

\begin{lemma}[see \cite{BdlV2010} Theorem 1,  and~\cite{Vinter2010Book} Theorem 9.3.1]\label{lemma:PontryaginNODelay} Let $(s_0,i_0)\in B$. 
Suppose $\hs,\,\hi,\,\hv:[0,T]\to \R$ is a feasible triplet and the control $\hv$ is an optimum for the optimal control problem $\cP^T_{\lambda_v,\lambda_i,i_M}$ in~\eqref{eq:OptimalControlProblemNoDelay}. Then there exist $p_0\in \{0,1\}$, $k
\in \R$, $\hp_s\in W^{1,1}([0,T])$, $\hp_i\in BV([0,T])$ and 
 a positive, bounded,  Borel measure $d\mu$ on $[0,T]$  
% $\mu\in BV_+([0,T])$ with $\mu(T^+)=0$  
such that
\begin{enumerate}[label=(\Alph*)]
%    \item \emph{(Non-Triviality)}
 %   \[
%(\hp_0,\hp_s(t),\hp_i(t))\neq (0,0,0)\;\;\;\forall \,t\in [0,T].
 %   \]
\item \emph{(adjoint equations)}
\[
\begin{aligned}
\begin{cases}
\hp'_s=-\beta\, \hi(\hp_i-\hp_s)+\hv\hp_s,\\
\hp'_i=-p_0\lambda_i+\gamma \hp_i-\beta\, \hs\,(\hp_i-\hp_s)-d\mu,
\end{cases}\\
\end{aligned}
\]
\item \emph{(non-degeneration)}
\[
p_0+d\mu([0,T])>0,
\]
\item \emph{(complementarity)}
\[
\int_{[0,{T}]} (\hi(t)-i_M) d\mu(t)=0,
\]
\item \emph{(transversality)}
\[
\begin{aligned}
\hp_s(T)=0 \;\;\text{ and }\;\;
\hp_i(T^+)=0,
\end{aligned}
\]
\item \emph{(Weierstrass condition)} 
\[
 \hv(t)\in \arg\min_{v\in V} \left(p_0\lambda_v-\hp_s(t)\hs(t)\right)v,
\]
implying that, setting 
\begin{equation}\label{def_phi}
\varphi(t):=p_0\lambda_v-\hp_s(t)\hs(t),
\end{equation}
we have
% \[
% v(t)=\begin{cases}
% v_M, \;\;\;&\text{if }t\in[0,T],\;\;\text{and }\;p_0\lambda_v-\hp_s(t)\hs(t)<0,\\
% 0, \;\;\;&\text{if }t\in[0,T],\;\;\text{and }\;p_0\lambda_v-\hp_s(t)\hs(t)>0,
% \end{cases}
% \]
\begin{equation}\label{eq:sfv}
v(t)=\begin{cases}
v_M &\text{if }\varphi(t)<0,\\
0 &\text{if }\varphi(t)>0,
\end{cases}
\end{equation}
for almost every $t\in [0,T]$.  The function $\varphi:[0,T]\to \R$ is called {\em switching function},
\item \emph{(constancy of the Hamiltonian)}
\[
\begin{aligned}
\widehat H(t)&:=H(\hs(t),\hi(t),\hv(t),\hp_s(t),\hp_i(t))\\&=p_0\lambda_i \hi(t)+p_0\lambda_v\hv(t)-\beta\hp_s(t) \hs(t)\hi(t)-\hp_s(t)\hs(t)\hv(t)+\beta \hp_i(t)\hs(t)\hi(t)-\gamma \hp_i(t)\hi(t)\\&\equiv k\quad\mbox{a.e. in }\, [0,T].
\end{aligned}
\] 
\end{enumerate}
\end{lemma}
% \textcolor{black}{Remark commenting the notation $d\mu$, motivated by historical reasons ($\mu$ as a discontinuous time-varying Lagrange multiplier, etc.) ? For example referring to~\cite{BdlVD2013}}

 In order to simplify the analysis, from now on we work under the assumption that the final time $T$ be large enough to ensure that herd immunity is reached.  

\begin{assumption}\label{assum:UnderHerdFInal} For the considered initial condition $(s_0,i_0)\in B$ in problem~\eqref{eq:OptimalControlProblemNoDelay}, the final time $T>0$ satisfies %$s_0>\frac{\gamma}{\beta}$ and  
\begin{equation}\label{eq:DefinitionT0}
T> \bar{t}(s_0,i_0):=\sup_{v\in\cV}\inf_{t\ge0}\{s^{s_0,i_0,v}(t)\le\frac{\gamma}{\beta}\}.
\end{equation}
\end{assumption}

\begin{remark}\label{rem_stgb}
The assumption requires that, for any control $v\in \cV$, the corresponding solution satisfies $s^{s_0,i_0,v}(T)<\frac{\gamma}{\beta}$, or, in other words, the herd-immunity threshold is  reached (in a strict sense) at the final time $T$, for any control $v\in \cV$.
In view of  \emph{(7)} of Theorem~\ref{th_euri}, Assumption~\ref{assum:UnderHerdFInal} is satisfied if $T>\frac{1}{\gamma i_0}(s_0-\frac{\gamma}{\beta})^+$.
\end{remark}

% This assumption is not restrictive, because, in view of  \emph{(7)} of Theorem~\ref{th_euri},  is always satisfied for any $T$ large enough.

% We note that, in view of Theorem~\ref{th_euri}, statement \emph{(6)}, Assumption~\ref{assum:UnderHerdFInal} is always satisfied, for a $T$ large enough, since $\lim\limits_{t\to \infty} s^{s_0,i_0,v}(t)<\frac{\gamma}{\beta}$, for all $(s_0,i_0)\in T$ and for any $v\in \cV$. {\color{black}Quindi poi ricordiamoci di rimuoverla.} 

\subsection{Characterization of the optimal controls}\label{subsec_charOC}

In the sequel we see how, under Assumption~\ref{assum:UnderHerdFInal}, the necessary conditions in Lemma~\ref{lemma:PontryaginNODelay} allow us to provide a qualitative description of an optimal control, for any initial condition.

{\color{black} In all statements of the present subsection, and unless explicitly stated otherwise, we assume that $(s_0,i_0)\in B$,  that  $\hs,\,\hi,\,\hv:[0,T]\to \R$ be a feasible triplet and $\hv$ be an optimal control for  problem $\cP^T_{\lambda_v,\lambda_i,i_M}$ in~\eqref{eq:OptimalControlProblemNoDelay}. 
Let us, moreover, suppose that  Assumption~\ref{assum:UnderHerdFInal} holds true and that  $p_0\in \{0,1\}$,  $\hp_s\in W^{1,1}([0,T])$, $\hp_i\in BV([0,T])$ and $d\mu$  are those provided in the statement of Lemma~\ref{lemma:PontryaginNODelay}.

}

We start with a lemma characterizing the structure of the measure $d\mu$.

\begin{lemma}[{\color{black}structure of $d\mu$}]\label{Lemma:Technical}
The following propositions hold.
\begin{enumerate}[leftmargin=*]
\item %$\mu$ has a  piecewise constant representative (denoted for simplicity by $\mu$), and 
Either one of the following alternative conditions hold:
\begin{enumerate}
\item $d\mu= 0$, implying $p_0=1$,
\item there exists $a>0$ such that $d\mu=a\delta_0$, implying $\hs(0)=s_0\leq \frac{\gamma}{\beta}$ and $\hi(0)=i_0=i_M$,
\item there exist $a>0$ and $t_0\in (0,T)$ such that 
% $\mu(t)=-a$ a.e. in $(0,t_0)$ and $\mu(t)=0$ a.e. in $(t_0,T)$. This implies 
$d\mu=a\delta_{t_0}$ and, moreover, $\hs(t_0)=\frac{\gamma}{\beta}$ and $\hi(t_0)=i_M$.
\end{enumerate}
\item {\color{black} For every $t\in[0,T]$ we have $[\hp_i(t)]=-d\mu(\{t\})\leq 0$; in particular,}   $d\mu(\{T\})=\hp_i(T)=0$.
\end{enumerate}
\end{lemma}

\begin{remark}\label{mu0n} A straightforward consequence of {\color{black} {\it (1)} of} Lemma \ref{Lemma:Technical}  is that $d\mu=0$ in a neighborhood of $T$; in other words, $\text{supp}(d\mu)\subset [0,T)$.  
\end{remark}

\begin{proof}
Let us prove ~\emph{(1)}. In the case $d\mu= 0$, by the non-degeneration property we have $p_0=1$.\\
Suppose then we are not in case \emph{(a)}, i.e., $d\mu \ne 0$.
Therefore, and by complementarity, we have $\varnothing\ne\text{supp}(d\mu)\subseteq \{t\in [0,T]\;\vert\;\hi(t)=i_M\}=: S_{i_M}$.

Let us preliminarily  note that $T\notin S_{i_M}$. Indeed, suppose by contradiction $\hi(T)= i_M$;  since by Assumption~\ref{assum:UnderHerdFInal} we have $\hs(T)<\frac{\gamma}{\beta}$, in this case we would have
\[
\hi'(T)=\beta \hs(T)i_M-\gamma i_M<0,
\]
which implies that there exists $\varepsilon>0$ such that $\hi(t)>i_M$ for all $t\in (T-\varepsilon, T)$, contradicting the feasibility of $(\hs,\hi,\hv)$. 

In the case $0\in S_{i_M}$, i.e., $\hi(0)=i_0=i_M$, since $(s_0,i_0)\in B$, we must have $\hs(0)=s_0\leq \frac{\gamma}{\beta}$ and thus $\hi$ is strictly decreasing. This implies that $0$ is the only instant in which $\hi(t)=i_M$ and thus we are in case~\emph{(b)}, by complementarity.

 It then remains only the case in which there exists $t_0\in (0,T)$ such that $t_0\in S_{i_M}$, i.e.,  $\hi(t_0)=i_M$. If this is the case, $t_0$ is an interior  global maximum point  of $\hi:[0,T]\to \R$ and thus $\hi'(t_0)=0$ ($\hi$ being differentiable by {\em(2)} of Theorem \ref{th_euri}), implying $\hs(t_0)=\frac{\gamma}{\beta}$. Since $\hs$ is strictly decreasing, this proves that $S_{i_M}=\{t_0\}$ and we are thus in case \emph{(c)}, by complementarity.

{\color{black}
The first part of assertion {\em(2)} is an immediate consequence of the co-state equations and the positivity of the measure $d\mu$.
 To prove the last part it is enough to  note that $T\notin \text{supp}(d\mu)$ and   $\hp_i(T^+)=0$ (see also Remark \ref{mu0n}).
}
\end{proof}

{\color{black}
A crucial role in the characterization of the optimal controls will be played by the following  auxiliary adjoint variable
$$
\eta:=\hp_i-\hp_s.
$$}

\begin{lemma}\label{lem_eta}
Under the additional assumption $(\lambda_i,\lambda_v)\neq (0,0)$,  
the following propositions hold. 
\vspace{-2ex}
\begin{enumerate}[leftmargin=*]
\item $\hp_s$ and $\varphi$ are absolutely continuous in $[0,T]$, while  $\hp_i$ and  $\eta$ are piecewise absolutely continuous with at most one discontinuity point $t_0\in[0,T)$. % in which we have $[\eta(t_0)]=[\hp_i(t_0)]=-d\mu(\{t_0\})\leq 0$.
\item $\widehat H(t)=k:=p_0\lambda_i \hi(T)\ \text{a.e. on } \,[0,T]$.
\item If $p_0=0$,  then a discontinuity  point $t_0\in[0,T)$ as in {\it (1)} exists 
and $v=v_M$ a.e.\ in $[0,t_0)$.
 \item {\color{black}Let $p_0=1$. If there exists $\widetilde t\in(0,T]$ such that $\eta(\wtt)>0$,   then $\eta>0$ in $[0,\wtt)$. } 
\item If $p_0=1$, there exists at most one $\widehat t\in [0,T]$ such that $\varphi(\widehat t\,)=0$ and, in such case, $\varphi(t)<0$ for all $t\in [0, \widehat t\,)$ and $\varphi(t)>0$ for all $t\in (\widehat t, T]$. Moreover, $\widehat t=T$ if and only if $\lambda_v=0$.
\end{enumerate} 
\end{lemma}

\begin{proof}
Assertion {\em(1)} is an immediate consequence of the co-state equations, the definition of $\varphi$ and the fact that, by Lemma~\ref{Lemma:Technical}, the measure $d\mu$ can concentrate in at most one point.

    Let us prove~\emph{(2)}. 
    {\color{black} By {\it(1)} and since $p_s(T)=0$ by \emph{(D)} of Lemma~\ref{lemma:PontryaginNODelay}, all functions appearing in the expression of the Hamiltonian, with the exception of the term $p_0\lambda_vv$, are continuous in $T$. Then, when the aforementioned term is zero, that is, }
    if $p_0=0$ or $\lambda_v=0$,
    by evaluating the Hamiltonian in $T$ and using 
{\color{black}$(D)$ and of Lemma \ref{lemma:PontryaginNODelay},  and} {\em(2)} of Lemma~\ref{Lemma:Technical},  we have $\widehat H(T)=p_0\lambda_i \hi(T)$.  The conclusion, in this case, follows then by $(F)$  of Lemma~\ref{lemma:PontryaginNODelay}. 
  Otherwise, we have $\varphi(T)=\lambda_v-\hp_s(T)\hs(T)=\lambda_v>0$ which, by continuity, holds also in a left neighborhood $(T-\varepsilon,T]$, $\varepsilon>0$. Thus, by property $(E)$ in Lemma~\ref{lemma:PontryaginNODelay}, we can take a representative $\hv=0$  in $(T-\varepsilon,T]$ so obtaining
 $$
\widehat H(t)=\lambda_i \hi(t)-\beta\hp_s(t) \hs(t)\hi(t)+\beta \hp_i(t)\hs(t)\hi(t)-\gamma \hp_i(t)\hi(t),
$$
and the claim follows by evaluating in $t=T$ and using $(F)$ of Lemma~\ref{lemma:PontryaginNODelay}.

Let us prove \emph{(3)}. Being $p_0=0$, by the non-degeneration property {\em(B)} of Lemma~\ref{lemma:PontryaginNODelay} we have $d\mu \ne 0$.  Thus, {\color{black} by {\it(1)} of Lemma \ref{Lemma:Technical},} a discontinuity point $t_0\in [0,T)$  exists {\color{black}and is unique.} In proving that $v=v_M$ a.e.\ in  $[0,t_0)$, we consider the only non-trivial case $t_0>0$, implying that we are in case~\emph{(1)(c)} of Lemma~\ref{Lemma:Technical}, {\color{black} that is,}  $t_0\in (0,T)$ and there exist $a>0$ and such that $d\mu=a\delta_{t_0}$. By linearity of the adjoint equations and since we have final homogeneous conditions, then  $\hp_s(t)=\hp_i(t)=0$ for all $t\in (t_0,T]$. This implies $\varphi(t)=0$ for all $t\in [t_0,T]$ {\color{black}(see \eqref{def_phi})}. 
Moreover, {\color{black} by {\it(2)} of Lemma \ref{Lemma:Technical}, we have  $[\hp_i(t_0)]=-d\mu(\{t_0\})=-a$ and, hence} $\hp_i(t_0^-)=a>0$.

It is easy to check {\color{black}that}
\begin{equation}\label{eq_phipBV}
\varphi' =\beta\, \hi \hs\hp_i\in BV([0,T])
\end{equation}
and $\varphi'(t_0^-)= \beta\, \hi(t_0^-) \hs(t_0^-)\hp_i(t_0^-)=
 \beta\, \hi(t_0^-) \hs(t_0^-)a>0
$. 
% Recalling that the derivative of $\varphi$ is given by~\eqref{eq:DerivativeVarPhi}, and that $\varphi'$ is piecewise continuous (with a unique discontinuity point at $t_0$), we can write 
% \[
% \begin{aligned}
% \varphi'(t_0^-)&=\lim_{t\to t_0^-} -\beta \hi(t)\varphi(t)+\beta\hi(t)\hs(t)\eta(t)\\&=\lim_{t\to t_0^-}\beta\, \hi(t) \hs(t)\hp_i(t)=\beta a i_M\hs(t_0)>0.
% \end{aligned}
% \]
Since $\varphi(t_0)=0$, this implies that there exists $\varepsilon>0$ such that $\varphi(t)<0$ for all $t\in (t_0-\varepsilon, t_0)$. 
The thesis will be proved by showing that $\varphi(t)<0$ for all $t\in [0,t_0)$,  {\color{black} and using \eqref{eq:sfv} to conclude}.

Suppose, by contradiction, that there exists $t\in[0,t_0)$ such that $\varphi(t)\geq 0$ and set
\[
t_1=\sup\{t\in [0,t_0)\;\vert\; \varphi(t)\geq0\}.
\]
Since we have proved that $\varphi(t)<0$ for all $t\in (t_0-\varepsilon,t_0)$ we have $t_1\in[0,t_0)$ and $\varphi(t_1)=0$,  by continuity. Moreover, recalling that $\varphi(t_1)=-\hp_s(t_1)\hs(t_1)$, this implies $\hp_s(t_1)=0$.
We now aim to prove that $\hp_i(t_1)>0$. The case $\hp_i(t_1)=0$ is excluded since we would have $\hp_s(t_1)=\hp_i(t_1)=0$  and thus by linearity of the adjoint equations, this would imply $\hp_i(t)=0$ for all $t\in [t_1,t_0)$, contradicting $\hp_i(t_0^-)=a>0$. Then, let us suppose by contradiction that $\hp_i(t_1)<0$. In this case, since $\varphi(t_1)=0$ the Hamiltonian is continuous at $t_1$ and reads
\begin{equation}\label{eq:HamiltonianinT2}
0=\widehat H(t_1)=\beta\eta(t_1) \hs(t_1)\hi(t_1)-\gamma \hp_i(t_1)\hi(t_1)=(\beta \hs(t_1)-\gamma) \hp_i(t_1)\hi(t_1),
\end{equation}
because  $\eta(t_1)=\hp_i(t_1)-\hp_s(t_1)=\hp_i(t_1)$. Now, since $\hs(t_0)=\frac{\gamma}{\beta}$ (recall that we are in case {\em(1)(c)} of Lemma \ref{Lemma:Technical}) $\hs$ is strictly decreasing (see~Theorem~\ref{th_euri}), 
 and $t_1<t_0$, we have $\hs(t_1)>\frac{\gamma}{\beta}$.
Since we are supposing $\hp_1(t_1)<0$, equality~\eqref{eq:HamiltonianinT2} implies $\widehat H(t_1)<0$, and thus leads to a contradiction. We have thus proved that $\hp_i(t_1)>0$.

Now, {\color{black}recalling~\eqref{eq_phipBV}} and since $\varphi(t_1)=0$, we have
% , since $\varphi(t)=-\hp_s(t)\hs(t)$, recalling~\eqref{eq:DerivativeVarPhi} we have
\[
\varphi'(t_1)=  \beta\hs(t_1)\hi(t_1)\hp_i(t_1)>0
\]
 which contradicts  $\varphi(t)<0$ in $(t_1,t_0)$. We thus have proved that $\varphi(t)<0$ for all $t\in [0,t_0)$ and then $\hv(t)=v_M$ for almost all $t\in [0,t_0)$.

Let us prove assertion {\it(4)}. 
We first observe that, with the introduced notation, we can more compactly write
\begin{equation}\label{eq:HamiltonianNoDelay}
\widehat H(t)=p_0\lambda_i \hi(t)+\varphi(t)\hv(t)+\beta\eta(t) \hs(t)\hi(t)-\gamma \hp_i(t)\hi(t).
\end{equation}
{\color{black}Let us, moreover, introduce the function
\begin{equation}\label{eq:DerofEta}
g(t):=-\lambda_i+\gamma \hp_i(t)+\beta( \hi(t)- \hs(t))\,\eta(t)- \hv(t)\hp_s(t)
\end{equation}
that is the absolutely continuous part of the derivative $\eta'=p_i'-p_s'$ with respect to the Lebesgue measure.}

{\color{black}
By contradiction, let us suppose that there exists $t\in [0,\wtt)$ such that $\eta(t)\leq0$.  Let us define
\begin{equation}\label{eq_deft1}
t_1:=\sup\{t\in [0,\wtt)\;\vert\;\eta(t)\leq0\}.
\end{equation}
By the assumption $\eta(\wtt)>0$, and since we are considering a left continuous representative,  $\eta$ is positive in a left neighborhood of $\wtt$. Thus, $t_1<\wtt$ and $\eta(t_1)\le0$ by left continuity. 
{\color{black}By {\it(2)} of Lemma \ref{Lemma:Technical} and the continuity of $p_s$, we have  $[\eta(t_1)]=[p_i(t_1)]=-d\mu(\{t_1\})\leq 0$.} Moreover, we note that the case $[\eta(t_1)]<0$ is excluded. Indeed, otherwise, we would have $\eta(t_1^+)<0$, which would imply $\eta(t_1)<0$ in a right neighborhood of $t_1$, against the definition of $t_1$.  }

 Thus, we have that $[\eta(t_1)]=0$.  Since the jump can occur in just one point,  this implies that $p_i$ and  $\eta$ are continuous in a small neighborhood of $t_1$ and, therefore, we have $\eta(t_1)=0$. 
 Noting that \eqref{eq:sfv} implies $\varphi\hv\le0$ almost everywhere and looking at the expression \eqref{eq:HamiltonianNoDelay} of the Hamiltonian \textcolor{black}{and using \emph{(1)} we have
 $$
 \lambda_i i(T)=\widehat{H}(t)\le
 \big(\lambda_i-\gamma \hp_i(t)\big)\hi(t){\color{black}+}\beta\eta(t) \hs(t)\hi(t)\quad\mbox{for a.a. }t\in[0,T].
 $$
 Since both sides of the inequality are continuous in $t_1$ and $\eta(t_1)=0$, then we get
\begin{equation}\label{litfle}
 \lambda_i i(T)\le
 \big(\lambda_i-\gamma \hp_i(t_1)\big)\hi(t_1).
\end{equation}
On the other hand, we can prove that
\begin{equation}\label{eq:l1gpi}
\lambda_i -\gamma \hp_i(t_1)\le0.
\end{equation}
Indeed, let us note that such term appears with the minus sign in the expression \eqref{eq:DerofEta} of $g$ and 
% To prove  term $\lambda_i-\gamma \hp_i(t_1)$ on the right hand side appears in the absolutely continuous part $g$ of the derivative of $\eta$ and it 
%  \textcolor{blue}{We want to prove that 
% \begin{equation}\label{eq:TechnicalinequProof}
% 0\leq-\lambda_i +\gamma \hp_i(t_1);
% \end{equation}
suppose by contradiction that 
$
0>-\lambda_i +\gamma \hp_i(t_1)$.
Since $\eta(t_1)=0$ and continuous in a neighborhood of  $t_1$ together with $p_i$,  there exists a  right neighborhood $I$ of $t_1$ such that 
\[
0>-\lambda_i +\gamma \hp_i(t)+\beta( \hi(t)- \hs(t))\,\eta(t),\;\forall \,t\in I.
\]
Moreover there exists a right neighborhood $J\subseteq I$ of $t_1$  such that 
\[
\begin{aligned}
0&>-\lambda_i +\gamma \hp_i(t)+\beta( \hi(t)- \hs(t))\,\eta(t)\\&\geq -\lambda_i +\gamma \hp_i(t)+\beta( \hi(t)- \hs(t))\,\eta(t)-{\color{black}\hv(t)\hp_s(t)}=g(t)\;\text{ for a.a.} \;t\in J,
\end{aligned}
\]
where the last inequality is trivial if {\color{black}$\hp_s(t_1)\geq0$}, and it holds if {\color{black}$\hp_s(t_1)< 0$} since in this case we have $\varphi(t_1)>0$ and thus $\hv(t)=0$ almost everywhere in a small-enough right neighborhood $J$, because $p_s$, and hence $\varphi$, are continuous.
Now, considering  $t\in J$,  we have
\[
\eta(t)=\int_{t_1}^tg(s)\;ds<0,
\]
contradicting the definition of $t_1$. We have thus proved~\eqref{eq:l1gpi}.  Let us continue by distinguish two cases. 
\begin{itemize}[leftmargin=0cm]
\item[\empty] {\em Case $\lambda_i>0$.}
Using~\eqref{eq:l1gpi} in \eqref{litfle}   would contradict $\lambda_ii(T)>0$; therefore, we have  proved that $\eta(t)>0$ for all $t\in [0,\wtt)$. \vspace{2ex}
\item[\empty] {\em Case $\lambda_i=0$.} 
 In this case, inequality \eqref{eq:l1gpi}  reads $-\gamma \hp_i(t_1)\leq 0$, implying $\hp_i(t_1)\geq 0$.
Moreover, multiplying by $\hi(t_1)$ this leads to $-\gamma \hp_i(t_1)\hi(t_1)\leq 0$. Recalling expression~\eqref{eq:HamiltonianNoDelay} and the fact that $\varphi \hv\leq 0$ a.e. in $[0,T]$, there exists a  right neighborhood $I_1$ of $t_1$ such that 
\[
0=\widehat H(t)\leq-\gamma \hp_i(t)\hi(t)\; \text{ a.e. in } I_1.
\]
By continuity of $\hp_i$ in $t_1$, this implies $\hp_i(t_1)\leq 0$, and thus $\hp_i(t_1)=0$. Since $\eta(t_1)=0$ this also  implies $\hp_s(t_1)=0$.
Since we have proved that $[\eta(t_1)]=[p_i(t_1)]=-d\mu(\{t_1\})=0$ there exists a right neighborhood $J_1\subseteq I_1$ of $t_1$, such that $J_1\cap \text{supp}(d\mu)=\varnothing$. Thus, since  $\hp_s(t_1)=\hp_i(t_1)=0$ and by linearity of the adjoint equations in the case $\lambda_i=0$, we have that $\hp_s(t)=\hp_i(t)=\eta(t)=0$ for all $t\in J_1$ contradicting the definition of $t_1$. Therefore, also in this case, it holds that $\eta(t)>0$ for all $t\in [0,\wtt)$.
\end{itemize}
}

Also in proving~\emph{(5)} we argue by distinguish the cases $\lambda_i>0$ and $\lambda_i=0$.
\begin{itemize}[leftmargin=0cm]
\item[\empty]{\em Case $\lambda_i>0$.}
% {\color{black} As a preliminary step, we prove that $\eta(t^-)>0$ for all $t\in [0,T)$.
% To this aim, } we introduce the function
% \begin{equation}\label{eq:DerofEta}
% g(t):=-\lambda_i+\gamma \hp_i(t)+\beta( \hi(t)- \hs(t))\,\eta(t)- \hv(t)\hp_s(t)
% \end{equation}
% that is the absolutely continuous part of the derivative of $\eta$ with respect to the Lebesgue measure.
By Lemma~\ref{Lemma:Technical} {\color{black} and {\it(D)} of Lemma~\ref{lemma:PontryaginNODelay}} we have $\hp_s(T)=\hp_i(T)=\eta(T)=0$ and, 
by   {\em(2)}, the functions  $\eta$,  $\hp_s$ and $\hp_i$ are absolutely continuous in a left neighborhood of $T$.
Moreover, since $\hv$ is bounded, there exists $\varepsilon>0$ such that $\eta'=g<0$  almost everywhere in  $(T-\varepsilon, T]$. Then, for every $t\in (T-\varepsilon, T)$ we have  $\eta(t)=-\int_t^{T}g(s)\;ds$,  which implies  $\eta(t)>0$.
{\color{black}By assertion {\it(3)} we have that $\eta>0$ in $[0,T)$.\\
To conclude the proof of {\it(5)} in the case $\lambda_i>0$,}  let us note that  
\begin{equation}\label{eq:DerivativeVarPhi}
\begin{aligned}
\varphi'&%= -\hp_s'\hs-\hp_s\hs'=+\beta \hi\eta\hs-\hv\hp_s\hs+\beta \hi\hs\hp_s+\hv\hs\hp_s
%\\&
= -\beta \hi\varphi+\beta\hi\hs\eta+\lambda_v\beta \hi. 
\end{aligned}
\end{equation}
Since $\varphi'$ does not explicitly depend on $\hv$, we have $\varphi'\in BV([0,T])$.
\begin{itemize}[leftmargin=0cm]
\item[\empty]{\em Subcase $\lambda_v>0$.}
Let us suppose first that $\lambda_v>0$, and thus $\varphi(T)=\lambda_v>0$.
% This implies that $\varphi(t)$ is solution of  a ``backward'' scalar Cauchy problem of the form
% \[
% \begin{cases}
% z'(t)=-a(t)z(t)+b(t),\\
% z(T)=\lambda_v,
% \end{cases}\quad t\in [0,T],
% \]
% with $a,b:[0,T]\to \R$ strictly positive, $a$ continuous and $b$ piecewise continuous (with at most one jump discontinuity).
Since  $\eta>0$ in  $[0,T)$, equation~\eqref{eq:DerivativeVarPhi}  implies that if there exists $\widehat t\in [0,T)$ such that $\varphi(\widehat t\,)=0$, then   $\varphi'(\widehat t^-\,)>0$ and $\varphi'(\widehat t^+\,)>0$. Thus, there exists $\delta>0$ such that $\varphi(t)<0$ for all $t\in (\widehat t-\delta, \widehat t)$ and $\varphi(t)>0$ for all $t\in (\widehat t,\widehat t+\delta)$ and, since $\varphi$ is continuous,  this implies that  $\widehat t$ is unique. In particular, if such $\widehat t\in [0,T)$ exists, we necessarily have $\varphi(t)<0$ for all $t\in [0,\widehat t\,)$ and $\varphi(t)>0$ for all $t\in (\widehat t, T]$, as claimed.
\item[\empty]{\em Subcase $\lambda_v=0$.} Supposing now that $\lambda_v=0$, we have $\varphi(T)=0$ and,  simplifying~\eqref{eq:DerivativeVarPhi},  it holds that
\begin{equation}\label{Eq:derivativePhiLambdaVZero}
\varphi'
= \beta\hi\hs\hp_i. 
\end{equation}
By continuity of $\hp_i$  in a left neighborhood of $T$ (see {\it(1)}), and since $\hp_i(T)=0$ and $\hp'_i(T)=-\lambda_i<0$ we have that there exists $\delta>0$ such that $\hp_i(t)>0$ for all $t\in (T-\delta, T)$. This in turns implies that $\varphi(t)<0$ for all $t\in (T-\delta, T)$, by~\eqref{Eq:derivativePhiLambdaVZero}. Then, the claim (i.e., $\varphi(t)<0$ for all $t\in [0,T)$) follows again from~\eqref{eq:DerivativeVarPhi} and the fact that $\eta>0$ in  $[0,T)$. Note that we have also proved that $\widehat t\,=T$ if and only if $\lambda_v=0$. 
\end{itemize}

\vspace{2ex}

\item[\empty]{\em Case $\lambda_i=0$}. In this case, $\lambda_v>0$, since we are supposing that $(\lambda_i,\lambda_v)\neq (0,0)$. {\color{black} Let us, now, distinguish the cases {\it(a)}, {\it(b)} and {\it(c)} of point {\it(1)} of Lemma \ref{Lemma:Technical}.} 

If $d\mu$ is of the form described in~{\it(a)} {\color{black}or} {\it(b)} of the mentioned lemma, the unique solution satisfying the adjoint equations in Lemma~\ref{lemma:PontryaginNODelay} is $\hp_s\equiv \hp_i\equiv0$. This implies $\varphi\equiv \lambda_v>0$  
{\color{black}and the claim {\it(5)} is proved in this case.} %{\color{green}(and thus $\hv\equiv 0$ almost everywhere) SERVE A QUALCOSA O POSSIAMO TOGLIERE?}

If we are in case~{\it(c)}, then there exists $t_0\in (0,T)$ such that $d\mu=a\delta_{t_0}$ with $a>0$. Arguing as in the previous case,  we have that  $\hp_s(t)=0$, $\hp_i(t)=0$,  for all $t\in (t_0,T]$.  This also implies $\varphi(t)=\lambda_v>0$ for all $t\in (t_0,T]$.
  {\color{black}  Moreover, by {\it(2)} of Lemma \ref{Lemma:Technical},} we have $[\eta(t_0)]=[\hp_i(t_0)]=-d\mu(\{t_0\})=-a< 0$ and thus there exists  $\varepsilon>0$ such that $\eta(t)>0$ for all $t\in (t_0-\varepsilon, t_0)$. 
{\color{black}
%Similarly to the previous case, we now  prove} 
By assertion {\it(3)}, we have 
that $\eta(t)>0$ for all $t\in [0,t_0)$.}
%  Suppose by contradiction that there exists $t\in [0,t_0)$ such that $\eta(t)\leq 0$. Consider thus
% \[
% t_1=\sup\{t\in [0,t_0)\;\vert\;\eta(t)\leq 0\}.
% \]
% By {\color{black} the positivity of $\eta$ in a left neighborhood of $t_0$ proved a few  lines above, we have} $t_1<t_0$ and by continuity of $\eta$ in $[0,t_0)$, we have $\eta(t_1)=0$. {Following the same steps of the previous case {\color{black}(immediately after \eqref{eq_deft1})} we can prove~\eqref{eq:l1gpi}, which in this case reads $-\gamma \hp_i(t_1)\leq 0$, implying $\hp_i(t_1)\geq 0$.
% Moreover, multiplying by $\hi(t_1)$ this leads to $-\gamma \hp_i(t_1)\hi(t_1)\leq 0$. Recalling expression~\eqref{eq:HamiltonianNoDelay} and the fact that $\varphi \hv\leq 0$ a.e. in $[0,T]$, there exists an open neighborhood $I$ of $t_1$ such that 
% \[
% 0=\widehat H(t)\leq-\gamma \hp_i(t)\hi(t)\; \text{ a.e. in } I.
% \]
% By continuity of $\hp_i$ in $t_1$, this implies $\hp_i(t_1)\leq 0$, and thus $\hp_i(t_1)=0$. Since $\eta(t_1)=0$ this also  implies $\hp_s(t_1)=0$.}
% Since $t_1<t_0$ and $\hp_s(t_1)=\hp_i(t_1)=0$, by linearity of the adjoint equations in the case $\lambda_i=0$ and by the fact that the measure $d\mu$ has support equal to $\{t_0\}$, we have that $\hp_s(t)=\hp_i(t)=0$ for all $t\in [0,t_0)$ leading to a contradiction, since we have already proved that $\hp_i(t_0^-)=\hp_i(t_0^+)-[\hp_i(t_0)]=0+a>0$.
% We have thus proved that $\eta(t)>0$ for all $t\in [0,t_0)$ 
 %
 {\color{black}
 Then, we can proceed exactly in the previous case (subcase $\lambda_v>0$) to prove that $\varphi(t)=0$ in at most one instant of time in $[0,t_0)$, and the conclusion follows by recalling that we have already proved that $\varphi\equiv \lambda_v>0$ in $(t_0,T]$. }
\end{itemize}
\end{proof}

{\color{black}

\begin{remark}\label{rem:EtaPositive}
    Let us note that, in the proof of point \emph{(4)} of Lemma~\ref{lem_eta}, we have proved that $\eta(t)>0$ for all $t\in [0,T)$ in the case $p_0=1$,  $\lambda_i>0$. 
\end{remark}

}
We now provide a complete characterization of the optimal controls of $\cP^T_{\lambda_v,\lambda_i,i_M}$.

\begin{theorem}
[structure of the optimal controls]\label{prop:OptimumNoDelay}
Let us make the additional assumption $(\lambda_i,\lambda_v)\neq (0,0)$.
The following propositions hold. 
\vspace{-2ex}
\begin{enumerate}[leftmargin=*]
\item[(A)] There exists $\tsw\in [0,T]$, the so-called  \emph{switching time}, such that $\hv$ is almost everywhere given by
\begin{equation}\label{eq:BangBangCOntrol}
\hv(t)=\begin{cases}
v_M&\text{if }t\in (0, \tsw),\\
0&\text{if }t\in (\tsw,T).
\end{cases}
\end{equation}
Moreover, $\tsw=T$ if and only if $\lambda_v=0$.
\item[(B)] If $\lambda_i=0$,  then the optimal control is unique and admits 
%for any optimal solution-control trajectory $(\hx,\hv)$ we have  
the following \emph{feedback representation}: %for the control:
$\hv(t)=w(\hx(t))$, with $w:B\to [0,v_M]$ defined by
\begin{equation}\label{eq:feedbackRepresentation}
w(x)=\begin{cases}
v_M&\text{if } x\in B\setminus A\\
0 &\text{if } x\in A. 
\end{cases}
\end{equation}
\item[(C)] It necessary holds that  $\hx(\tsw)\in A$. If $\hx(t_\star)\in\partial A$ then $\hv$ admits the feedback representation \eqref{eq:feedbackRepresentation}. 
If $\hx(t_\star)\in{\rm Int}(A)$ then $\hi(t)<i_M$ for every $t\ge t_\star$.
\end{enumerate} 
\end{theorem}

% {\color{black} The following lemma will be useful in the proof.

% \begin{lemma}
% If there exists $t_0\in(0,T]$ such that $\eta(t_0)>0$  then $\eta(t)>0$ in $[0,t_0)$.     
% \end{lemma}

% \begin{proof}
    
% \end{proof}

% }

\begin{proof}

Let us prove~\emph{(A)}. If $p_0=1$, let us consider $\widehat t\in [0,T]$ as in assertion~\emph{(5)} of Lemma~\ref{lem_eta}. Then, the claim follows by taking $\tsw=\widehat t$ and using the continuity of $\varphi$ and {\color{black}the consequence \eqref{eq:sfv} of} the Weierstrass condition \emph{(E)} of Lemma~\ref{lemma:PontryaginNODelay}.
% Indeed, in~\emph{(2)} we have proved that, given the time instant $\widehat t\in [0,T]$ such that $\varphi(\widehat t)=0$ (if any), we have $\varphi>0$ in $(\widehat t, T]$ and $\varphi<0$ in $[0,\widehat t\,)$, and this, recalling the Weierstrass condition (Item \emph{(E)} in Lemma~\ref{lemma:PontryaginNODelay}) concludes the proof.

In the case $p_0=0$, by   \emph{(3)} of Lemma~\ref{lem_eta}, we have $\hv(t)=v_M$ for almost all $t\in [0,t_0)$, where $t_0\in [0,T)$ is the (unique) discontinuity point of $\eta$ and $p_i$.
{\color{black}It remains to}  characterize the behavior of the control $\hv$ in the interval $(t_0,T]$. 
Since in \emph{(1)} of Lemma~\ref{Lemma:Technical} we have proved that $\hs(t_0)\leq \frac{\gamma}{\beta}$ and $\hi(t_0)=i_M$ and recalling from Theorem~\ref{th_euri} that $\hs$ is strictly decreasing, we have that $\hs(t)<\frac{\gamma}{\beta}$ for all $t\in (t_0,T]$. Then, we also have $\hi'(t)=\beta\hs(t)\hi(t)-\gamma \hi(t)\leq (\beta\frac{\gamma}{\beta}-\gamma)\hi(t)<0$ for all $t\in (t_0,T]$. Since $\hi(t_0)=i_M$, this implies $\hi(t)<i_M$ for all $(t_0,T]$. {\color{black} Let us now distinguish the two cases $\lambda_i=0$ and $\lambda_i>0$.}
\begin{itemize}[leftmargin=0cm]
    \item[\empty] {\em Case
$\lambda_i=0$.} In this case, at $t=t_0$ we are in the set $A=A_0$ (see Definition \ref{def_viableset} and Theorem \ref{lemma:ViabilityDelayFree}). Then $\hv\equiv0$ in $(t_0,T)$ is trivially optimal.  \vspace{2ex}
\item[\empty] {\em Case $\lambda_i>0$.}  We observe that for any $\varepsilon>0$ and 
$t_\varepsilon:=t_0+\varepsilon \in (t_0, T)$, 
the principle of optimality {\color{black} (see Proposition \ref{Prop:PrincipleOfOptim})} implies that the control $\hv(\,\cdot+t_\varepsilon)\in L^\infty([0,T-t_\varepsilon], V)$ is an optimum for problem~\eqref{eq:OptimalControlProblemNoDelay} with initial condition  $(s(0),i(0))=(\hs(t_\varepsilon),\hi(t_\varepsilon))$ for the time-horizon $[0,T_{\varepsilon}]$,  with $T_{\varepsilon}=T-t_\varepsilon$. Since  $\hi(t)<i_M$ for all $t\in [t_\varepsilon,T]$, considering the necessary conditions of  Lemma~\ref{lemma:PontryaginNODelay} for this new optimal control problem, we would trivially obtain $d\mu_\varepsilon= 0$ (implying $p_{0,\varepsilon}=1$), and thus, again by assertion \emph{(5)} of Lemma~\ref{lem_eta}, we necessarily have that $\hv$ in $[t_\varepsilon, T]$ is of the bang-bang form, i.e., there exists a $t_{\star,\varepsilon}\in [t_\varepsilon, T]$ such that the expression
\[
\hv(t)=\begin{cases}
v_M\;\;\;&\text{if}\;t\in (t_\varepsilon,t_{\star,\varepsilon}),\\
0\;\;\;&\text{if}\;t\in ( t_{\star,\varepsilon},T),
\end{cases}
\]
holds for almost all $t\in [t_\varepsilon, T]$. 
 By letting $\varepsilon\to 0$, we have that there exists $\tsw\ge t_0$ such that (up to a subsequence) $t_{\star,\varepsilon}\to \tsw$ and we obtain that 
 \[
\hv(t)=\begin{cases}
v_M\;\;\;&\text{for a.e.\ }t\in (t_0,\tsw),\\
0\;\;\;&\text{for a.e.\ }t\in (\tsw,T).
\end{cases}
\]
Now we can conclude by recalling that, by  assertion \emph{(3)} in Lemma~\ref{lem_eta}, $\hv(t)=v_M$ for almost all $t\in [0,t_0)$. We also note that, by \emph{(5)} of Lemma~\ref{lem_eta}, in the preceding construction the sequence $t_{\star,\varepsilon}$ is constant and equal to $ t_\star=T$ if and only if $\lambda_v=0$.
\end{itemize}

Let us now prove~\emph{(B)}. By~\emph{(A)} the optimal control is of the \emph{bang-bang} form as in~\eqref{eq:BangBangCOntrol}. If $(s_0,i_0)\in A$, by forward invariance proved in Theorem~\ref{lemma:ViabilityDelayFree}, we have that the control $v\equiv 0$ is feasible, and moreover it is straightforwardly the unique minimum, since $\lambda_i=0$.
Let us then consider a point $(s_0,i_0)\in B\setminus A$ and suppose by contradiction that the switching time $\tsw \in [0,T)$ is such that $\hx(\tsw)\notin A$. In this case, we have that $\hs(\tsw)>\frac{\gamma}{\beta}$ and $\hv(t)=0$ for a.a.\ $t\in (\tsw,T]$. Since, by Assumption \ref{assum:UnderHerdFInal}, $\hs(T)< \frac{\gamma}{\beta}$, there exists $t_{\frac{\gamma}{\beta}}\in (\tsw, T)$ such that $\hs(t_{\frac{\gamma}{\beta}})=\frac{\gamma}{\beta}$. But, then, we would have reached the no-effort set $A=\cA_0$ with control $0$ in contradiction with $\hx(\tsw)\notin A$ (recall Theorem~\ref{lemma:ViabilityDelayFree}).
Then we have proved that $\hx(\tsw)\in A$ and the proof of {\em(B)} is concluded, by noting that the aforementioned feedback representation directly implies that the optimal control is unique.  

To prove \emph{(C)} we first observe that also in the general case we necessarily have $\hx(\tsw)\in A$ 
because switching outside the set $A$ leads to infeasibility, and thus to a contradiction. 

If $\hx(t_\star)\in\partial A$ then the trajectory still remains inside $A$ by forward invariance and the optimal control admits the claimed feedback representation. 

If, instead, $\hx(t_\star)\in{\rm Int}(A)$, to prove that $i(t)<i_M$ for all $t>t_\star$ we distinguish the following cases.
\begin{enumerate}[leftmargin=0cm]
\item[\empty] If $\hs(t_\star)\le\frac{\gamma}{\beta}$, then $\hi(t_\star)<i_M$ and $\hi$ is strictly decreasing in $(t_\star,T]$, and so we are done.\vspace{2ex}
\item[\empty] If $\hs(t_\star)>\frac{\gamma}{\beta}$,  then we are strictly under the curve $\Gamma_A$ which has maximum value $i_M$. The claim follows by the fact that, with control zero, we remain strictly under the curve (see Remark~\ref{rem_cd}).
\end{enumerate}
\vspace{-3ex}
\end{proof}

% The following lemma is useful in what follows.

{
\begin{remark}[particular cases] \label{rem:ParticularAndUniqueness}
In the case $\lambda_i=0$, the feedback representation provided in \emph{(B)} of Theorem~\ref{prop:OptimumNoDelay},  implies that 
%\begin{itemize}[leftmargin=*]
 %   \item the optimal control is \emph{unique} for any feasible initial condition $(s_0,i_0)\in B$;
    %\item 
    at the \emph{switching time} $\tsw$ we have  $(s(\tsw),i(\tsw))\in \partial A$; hence $\tsw$ occurs before reaching herd immunity, that is, ${\color{black}\tsw}\le t_{\frac{\gamma}{\beta}}\le \frac{1}{\gamma i_0}\big(s_0-\frac{\gamma}{\beta})^+$ (see \emph{(7)} of Theorem~\ref{th_euri}).
Moreover, the \emph{switching point} $(s_\star,i_\star):=(s(\tsw),i(\tsw))$, i.e., the point of the trajectory in which the optimal control is discontinuous, lies on the intersection of the trajectory of~\eqref{eq:SIR2} with $v\equiv v_M$ starting at $(s_0,i_0)$ (plotted in red in Figure~\ref{optl10_fig} 
 {\color{black} and given by \eqref{eq_siid} with $v_0=v_M$}), with the curve $\Gamma_A$ defined in {\color{black}equation \eqref{eq:GammaBar} of}  Theorem~\ref{lemma:ViabilityDelayFree} (represented in green in Figure~~\ref{optl10_fig}). It can thus be easily computed {\color{black}by solving the system of equations of the two curves}, for any initial condition $(s_0,i_0)\in B$, and it {\color{black} turns out to be} equal to 
$$
i_\star=i_0{\rm e}^{\frac{\beta}{v_M}(s_0-\frac{\gamma}{\beta}+i_0-i_M-\frac{\gamma}{\beta}\log\frac{\beta s_0}{\gamma})},\quad 
s_\star=-\frac{\gamma}{\beta}{\color{black}W_{-1}}\Big(-\frac{1}{{\rm e}^{1+\frac{\beta}{\gamma}(i_M-i_\star)}}\Big), 
$$
where {\color{black}$W_{-1}$ is the inverse of the restriction to $(-\infty,-1)$ of the function $f(x)=x{\rm e}^x$, which is the so-called}  Lambert's function. We plotted a qualitative representation of the optimal solution in the case $\lambda_i=0$, together with the switching point $(s_\star,i_\star)$, in  Figure~\ref{optl10_fig}.

\begin{figure}[h!]
\begin{center}
\includegraphics[width=0.95
\textwidth]{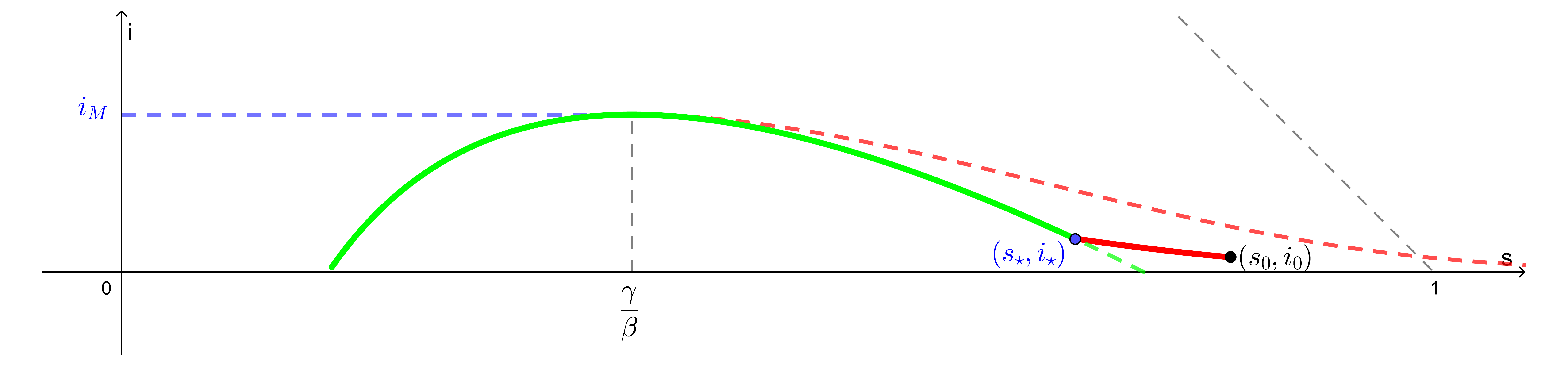}
\end{center}
\caption{The optimal strategy in the case $\lambda_i=0$}\label{optl10_fig}
\end{figure}

In the case $\lambda_v=0$, as proved in~\emph{(A)} of Theorem~\ref{prop:OptimumNoDelay} we have $t_\star=T$ and, thus, the control $\hv\equiv v_M$ is the (unique) optimum.
  In other words, when the cost does not explicitly depend on the control input, the choice $\hv=v_M$ minimizes the integral $\int_0^{T} \hi(t)\;ds$.  This reflects the intuitive idea that, if the vaccination action has no cost, the optimal strategy is to always keep the vaccination at its maximum admissible rate. 

In the general case, where both $\lambda_i>0$ and $\lambda_v>0$, providing an explicit characterization of the switching point/time is  challenging and will be discussed in the rest of this subsection.
% We can anyhow note that, by~\emph{(5)} of Proposition~\ref{prop:OptimumNoDelay},  we have the following  dichotomy.
% \begin{enumerate}[leftmargin=*]
%     \item $\hx(\tsw)\in \partial A$, and in this case the optimal control and the optimal trajectory coincide with the ones of the case~$\lambda_i=0$. This is numerically observed (see Figure~\ref{l1=0}) when  the relative cost of the infected population ($\lambda_i$), is small with respect to the relative cost ($\lambda_v$) of the vaccination policy.
% \item $\hx(\tsw)\in \text{Int}(A)$. In this case, we have that $\max_{t\in [0,T]}\hi(t)<i_M$, i.e., the maximum admissible level  $i_M$ of infected population is never reached. This happens when $\lambda_i$ is (relatively) large with respect to $\lambda_v$. Intuitively, in this case, it is convenient to keep the vaccination effort for a larger amount of time, and thus to keep the infected population strictly below the warning level. This phenomenon is shown in Figure \ref{l1=l2=1}.
% \end{enumerate}

\end{remark}
In the following statements we provide some further qualitative observations. First, we prove that if $\lambda_i/\lambda_v\le\beta$ then the vaccination intervention must cease  before reaching herd immunity.

\begin{prop}\label{pr_l2lebl1}
Let $\lambda_i,\lambda_v>0$ and suppose that Assumption~\ref{assum:UnderHerdFInal} holds true.   Let $v$ be an optimal control for problem in~\eqref{eq:OptimalControlProblemNoDelay} with a given $(s_0,i_0)\in B$ and $s,i:[0,T]\to \R$ be the corresponding trajectory.   Moreover, let us assume that the switching time $\tsw$ is strictly positive.  If  $\lambda_i\leq \beta\lambda_v$ then $\hs(\tsw)\geq \frac{\gamma}{\beta}$.
\end{prop}

\begin{remark}
In Proposition \ref{pr_l2lebl1}  we are assuming that, for the considered initial conditions, we have a non-trivial switching time $\tsw>0$. % \in (0,T)$. 
This, under Assumption~\ref{assum:UnderHerdFInal}, is straightforwardly satisfied if $(s_0,i_0)\in B\setminus A$ (see~\emph{(C)} of Theorem~\ref{prop:OptimumNoDelay}). On the other hand, this could be the case also if $(s_0,i_0)\in A$, if $\lambda_i>0$ is large enough.   
Figures \ref{optl2_betal1_fig} and \ref{optl2big_fig} show the qualitative behavior of the optimal strategies that can be expected in view of Proposition \ref{pr_l2lebl1} in the cases in which  $\lambda_i$ is, 
 respectively, small and large enough (with respect to $\lambda_v$). 
\end{remark}

\begin{figure}[h!]
\begin{center}
\includegraphics[width=0.95
\textwidth]{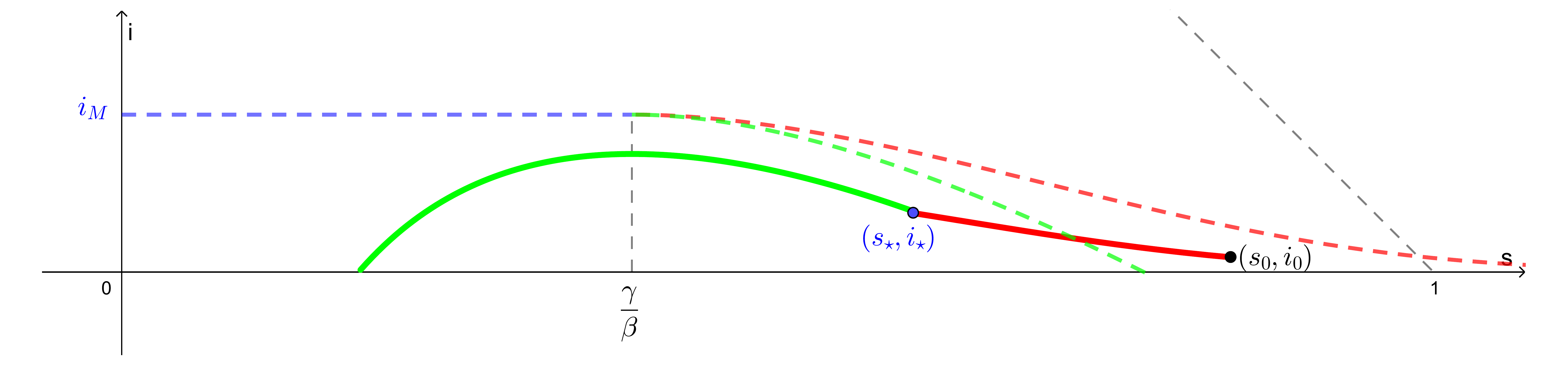}
\end{center}
\caption{Optimal strategy in the case $0<\lambda_i<\beta\lambda_v$}\label{optl2_betal1_fig}
\end{figure}

\begin{figure}[h!]
\begin{center}
\includegraphics[width=0.95
\textwidth]{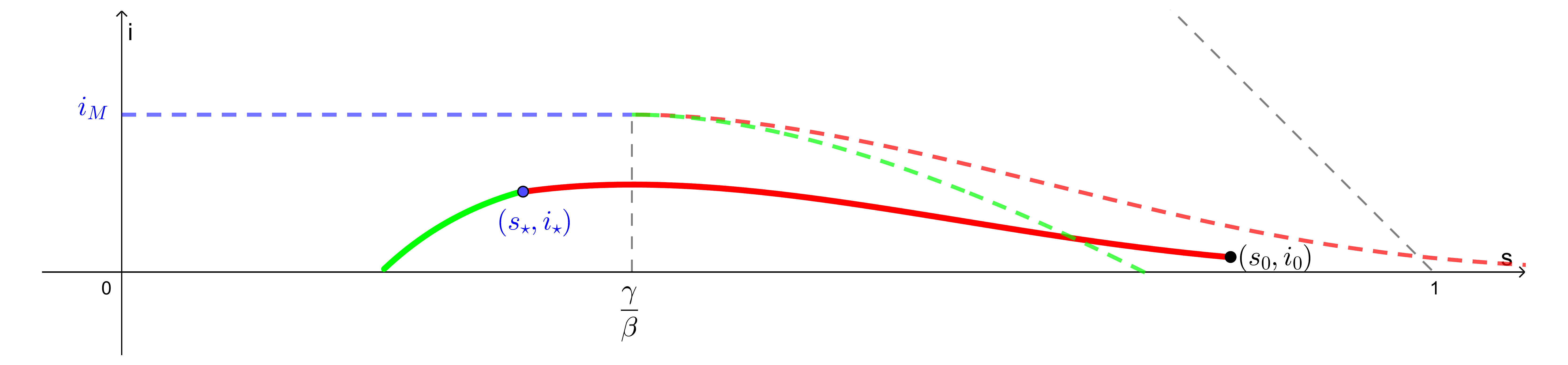}
\end{center}
\caption{Expected optimal strategy in the case $\lambda_i>\beta\lambda_v$ and large    enough\label{optl2big_fig}}
\end{figure}

\begin{proof}
We prove the contrapositive implication, that is,  $\hs(\tsw)<\frac{\gamma}{\beta}\;\Rightarrow\;\lambda_i>\lambda_v\beta$.
Let us suppose, without loss of generality, that, for the considered representative choice of $\hp_s,\hp_i,\hv:[0,T]\to \R$, the Hamiltonian $\widehat H$ be constant on $[0,T]$ (everywhere, not only almost everywhere). 
Let us argue by cases.
\begin{itemize}[leftmargin=0cm]
    \item[\empty] \emph{Case $p_0=1$.} By \emph{(2)} of Lemma~\ref{lem_eta}, we have
\[
0<\lambda_i \hi(T)=\widehat H(\tsw).
\]
In $\tsw$, by definition, we have $\varphi(\tsw)=0$, and, \textcolor{black}{since $\lambda_v\neq 0$ we have $\tsw\in (0,T)$ (see assertion~\emph{(A)} in Theorem~\ref{prop:OptimumNoDelay})}. Thus, \textcolor{black}{using the expression~\eqref{eq:HamiltonianNoDelay} of the Hamiltonian}, the assumption $\hs(\tsw)<\frac{\gamma}{\beta}$ \textcolor{black}{and the fact that $\eta(t)>0$ for all $t\in [0,T)$ (see Remark~\ref{rem:EtaPositive})}, we obtain
\[
\begin{aligned}
0<\widehat H(\tsw)&={\color{black}\lambda_i \hi(\tsw)-\gamma \hp_i(\tsw)\hi(\tsw)+\eta(\tsw)\beta \hs(\tsw)\hi(\tsw)}\\&{\color{black}< \lambda_i \hi(\tsw)-\gamma \hp_i(\tsw)\hi(\tsw)+\gamma\eta(\tsw) \hi(\tsw)}=\lambda_i \hi(\tsw)-\gamma \hp_s(\tsw)\hi(\tsw).
\end{aligned}
\]
% \[
% \begin{aligned}
% 0<\widehat H(\tsw)&=  \lambda_i \hi(\tsw)+\beta \eta(\tsw)\hs(\tsw)\hi(\tsw)-\gamma \hp_i(\tsw)\hi(\tsw)<\lambda_i \hi(\tsw)+\beta \eta(\tsw)\frac{\gamma}{\beta}\hi(\tsw)-\gamma \hp_i(\tsw)\hi(\tsw)\\&=\lambda_i \hi(\tsw)-\gamma \hp_s(\tsw)\hi(\tsw).
% \end{aligned}
% \]
Now, recalling that $0=\varphi(\tsw)=\lambda_v-\hp_s(\tsw)\hs(\tsw)$, we have that $\hp_s(\tsw)=\frac{\lambda_v}{\hs(\tsw)}$, and thus 
\[
0<\lambda_i \hi(\tsw)-\gamma\frac{\lambda_v}{\hs(\tsw)} \hi(\tsw)<\lambda_i \hi(\tsw)-\gamma\beta\frac{\lambda_v}{\gamma}\hi(\tsw)=(\lambda_i-\beta\lambda_v)\hi(\tsw).
\]
This implies $0<\lambda_i-\beta\lambda_v$, i.e., $\lambda_i>\beta\lambda_v$ as required.\vspace{2ex}
\item[\empty] \emph{Case $p_0=0$.} {By {\it(1)} of Lemma~\ref{Lemma:Technical} there exists $t_0\in [0,T)$ such that $\text{supp}(d\mu)=\{t_0\}$,  $\hs(t_0)\leq \frac{\gamma}{\beta}$ and $\hi(t_0)=i_M$, with  $\hs(t_0)<\frac{\gamma}{\beta}$ only if $t_0=0$.
By~\emph{(3)} of Lemma~\ref{lem_eta} we have  $\hv(t)=v_M$ for almost all $t\in [0,t_0)$.
Since we are supposing $\tsw>0$ is such that $\hs(\tsw)<\frac{\gamma}{\beta}$, we necessarily have that 
$\tsw>t_0$.

Since  $\hs$ is strictly decreasing and $\hs(t_0)\leq\frac{\gamma}{\beta}$, we have $\hs(t)<\frac{\gamma}{\beta}$ for all $t\in (t_0,T]$. Since $i(t_0)=i_M$, this in turns implies that $i(t)<i_M$ for all $t\in (t_0,T]$. {\color{black}
We can now consider the   optimal control problem, $\cP^{T-t_1,\hs(t_1),\hi(t_1)}_{\lambda_v,\lambda_i,i_M}$, with a new initial time $t_1\in (t_0,\tsw)$ and a new initial condition $(\hs(t_1),\hi(t_1))\in B$.}
By the principle of optimality (see Proposition~\ref{Prop:PrincipleOfOptim}), $\hv(\,\cdot\,+t_1)$ is an optimum also for this new optimal control problem, and thus satisfies the conditions of Theorem~\ref{prop:OptimumNoDelay} in the non-degenerate case (i.e., with $p_0=1$). 
 {\color{black} This is because  $i(t)<i_M$ for all $t\in [t_1,T]$ which implies that (by complementarity) the multiplicator $d\mu$ for this new problem is $0$, which in turn implies $p_0=1$ by \emph{(1)(a)} of Lemma~\ref{Lemma:Technical}.} Then, by arguing as in the previous case, it can be proved that  $\lambda_i>\beta\lambda_v$, as required. }
\end{itemize}
\vspace{-4ex}
\end{proof}

{
In the next statement, we prove that, when $\lambda_i$ is small enough depending on  $\lambda_v$, any optimal control  of $\cP^T_{\lambda_v,\lambda_i,i_M}$ coincides with the optimal control defined in~\emph{(B)} of Theorem~\ref{prop:OptimumNoDelay}, corresponding to the case $\lambda_i=0$. } This, in particular, proves the uniqueness of the optimal control of $\cP^T_{\lambda_v,\lambda_i,i_M}$ in the case in which the cost of the vaccination program is ``prevalent'' with respect to the cost of treatment of infected individuals.

\begin{prop}\label{prop:LambdaSmallEnough}
% Consider any $\gamma,\beta, v_M>0$, any $i_M\in (0,1]$ any 
Let $(s_0,i_0)\in B$ and $T>0$ satisfying Assumption~\ref{assum:UnderHerdFInal}. The following equivalent propositions hold.
\begin{enumerate}[leftmargin=*]
\item
For every $\lambda_v>0$ there exists $\overline \lambda_i>0$ such that,  for any $\lambda_i\leq \overline \lambda_i$, 
the optimal controls for $\cP^{T}_{\lambda_v,\lambda_i,i_M}$ coincide with 
the (unique) optimal control  for $\cP^T_{1,0,i_M}$.
\item For every $\lambda_i>0$ there exists $\overline \lambda_v>0$  such that, for any $\lambda_v\geq \overline \lambda_v$, the  optimal controls for $\cP^{T}_{\lambda_v,\lambda_i,i_M}$ coincide with 
the (unique) optimal control  for $\cP^T_{1,0,i_M}$.
\end{enumerate}
\end{prop}
\begin{proof}
The equivalence of \emph{(1)} and \emph{(2)} trivially follows \textcolor{black}{from} the fact that, for any $\lambda_v,\lambda_i\geq 0$, if $v$ is optimal for $\cP^T_{\lambda_v,\lambda_i,i_M}$, it is optimal also for $\cP^T_{\omega \lambda_v,\omega\lambda_i,i_M}$, for any $\omega>0$.

Let us prove~\emph{(2)}. Fixing $\lambda_i>0$, for any $\lambda_v>0$ we  consider any $v_{\lambda_v}\in \cV_T$ optimal for $\cP^T_{\lambda_v,\lambda_i,i_M}$, and denote by $x^{v_{\lambda_v}}:[0,T]\to B$ the corresponding trajectory.  By Assumption~\ref{assum:UnderHerdFInal}, we can consider $t_A:=t_A(\lambda_v)\in [0,T)$, the minimal time such that $x^{v_{\lambda_v}}(t_A(\lambda_v))\in A$.
By Theorem~\ref{prop:OptimumNoDelay} (assertion  {\it(A)})  we know that $v_{\lambda_v}$ necessarily has a $v_M\,$-$\,0$ bang-bang form while assertion  {\it(C)} of the same proposition implies  $v^{\lambda_v}(t)=v_M$ a.e. in $[0,t_A]$. As a first consequence, this implies that $t_A$ is, in fact,  independent of $\lambda_v$. Moreover, by the 
principle of optimality, $v_{\lambda_v}(\,\cdot\,+t_A)$ is optimal for $\cP^{T-t_A}_{\lambda_v,\lambda_i,i_M}$ with initial condition equal to $x^{v_{\lambda_v}}(t_A)$.

Since this ``new'' initial condition lies in $A$, which is forward invariant, $v_{\lambda_v}(\,\cdot\,+t_A)$ is optimal also for the \emph{unconstrained} problem $\cP^{T-t_A}_{\lambda_v,\lambda_i,1}$. Let us consider the corresponding adjoint state $(p_s,p_i):[0,T-t_A]\to\R^2$ satisfying the conditions in Lemma~\ref{lemma:PontryaginNODelay}; since the problem is unconstrained, we can assume without loss of generality that $d\mu=0$  and, hence,  $p_0=1$. 
We claim that there exists a $\overline \lambda_v$ large enough such that $v_{\lambda_v}(t)=0$ a.e.\ in $[t_A,T]$, for any $\lambda_v\geq \overline \lambda_v$.
We suppose, by contradiction, that for any $\lambda_v>0$ there exists a non-trivial switching time $\tsw\in (t_A,T)$ (i.e., $v_{\lambda_v}=v_M1_{[0,\tsw]}$).  At such time, again by Lemma~\ref{lemma:PontryaginNODelay}, we must have
\begin{equation}\label{eq:ConditionSwitchingg}
0=\varphi(\tsw)={\color{black}\lambda_v-p_s(\tsw-t_A)\hs^{v_{\lambda_v}}(\tsw)} \Leftrightarrow \hs^{v_{\lambda_v}}(\tsw)=\frac{\lambda_v}{p_s(\tsw-t_A)}.
\end{equation}
Since $(p_s,p_i)$ is solution to the adjoint equation in~(A) of Lemma~\ref{lemma:PontryaginNODelay}, which is a non-autonomous affine system with bounded coefficients,  then $p_s$ is Lipschitz continuous with constant $L>0$ depending on $(s_0,i_0)$ and $\lambda_i$ (but independent of  $\lambda_v$). Since $p_s(T-t_A)=0$, we thus have  $|p_s(t-t_A)|\leq L(T-t_A)$, for all $t\in [t_A,T]$. Thus, from~\eqref{eq:ConditionSwitchingg} we obtain 
\[
\hs^{v_{\lambda_v}}(\tsw)\geq\frac{\lambda_v}{L(T-t_A)}
\]
for every $\lambda_v>0$.  Since $t_A$ is a constant independent of $\lambda_v$, this is in contradiction with the fact that $\hs^{v_{\lambda_v}}(\tsw)\leq s_0$, for every $\lambda_v>0$. We have, in particular, proved that for any $\lambda_v\geq s_0L(T-t_A)$, $t_\star=t_A$, and thus $v_{\lambda_v}\equiv v_0$, where $v_0$ is the unique optimal control of $\cP^T_{1,0,i_M}$ characterized in~\emph{(A)} of Theorem~\ref{prop:OptimumNoDelay}, as claimed in the statement. 
\end{proof}

\subsection{Discussion and remarks about  uniqueness}\label{ss_uniqueness}

Under rather general assumptions, the existence of  optimal controls for problem  $\cP^T_{\lambda_v,\lambda_i, i_M}$ has been  proved in Corollary \ref{cor_ocfc}. Then, Theorem \ref{prop:OptimumNoDelay} ensures that such optima necessarily have a bang-bang structure $v_M\,$-$\,0$.
A natural question is 
to establish the \textcolor{black}{uniqueness} (or not) of such optima.
To our best knowledge this problem is still open in its general formulation, while 
a partial answer has been  given in Proposition \ref{prop:LambdaSmallEnough}. This subsection is devoted to further analyze the uniqueness problem.

First of all we see that  the uniqueness for the constrained problem can be reduced to the same property for 
the unconstrained one.

Given a control $v\in\cV$  and  the corresponding trajectory $x^v$, we introduce  the {\em reaching time} of the set $A$ as 
$$
t_A^v:=\inf\{t\in[0,T]\ \vert\ x^v(t)\in A\}
$$
with the convention $\inf\varnothing=\infty$. 

% \begin{theorem}
%  Let $\lambda_v\geq 0$ and $\lambda_i\geq 0$ be such that $(\lambda_v,\lambda_i)\neq (0,0)$, and $i_M\in(0,1]$ and $(s_0,i_0)\in B$ and $T>\bar{t}(s_0,i_0)$. Assume that  for any choice of the initial conditions $(s_1,i_1)$ with $s_1\ge s_0$ and $i_1\le i_0$ and any $T\ge T_1>\bar{t}(s_1,i_1)$ the unconstrained problem $\cP^{T_1}_{\lambda_v,\lambda_i,1}$ admits a unique optimal control.  Then,  
%  problem  $\cP^T_{\lambda_v,\lambda_i, i_M}$  admits a unique optimal control.  
% \end{theorem}

\begin{prop}
 Let $\lambda_v\geq 0$ and $\lambda_i\geq 0$ be such that $(\lambda_v,\lambda_i)\neq (0,0)$,  $i_M\in(0,1]$,  $(s_0,i_0)\in B$ and $T>\bar{t}(s_0,i_0)$.
 Let  $
t_A^{v_M}$  be the reaching time of the set $A$ with constant control $v_M$.
 Assume that  for  the initial condition   $(s_1,i_1):=\big(s(t_A^{v_M}),i(t_A^{v_M})\big )$   the unconstrained problem $\cP^{T-t_A^{v_M},s_1,i_1 }_{\lambda_v,\lambda_i,1}$ admit{\color{black}s} a unique optimal control.  Then,  
 problem  $\cP^{T,s_0,i_0}_{\lambda_v,\lambda_i, i_M}$  admits a unique optimal control.  
\end{prop}

\begin{proof} 
Let $v$ be an optimal control for $\cP^{T,s_0,i_0}_{\lambda_v,\lambda_i, i_M}$ and $x^v$ be the corresponding trajectory which must reach the set $A$ because $T$ satisfies Assumption~\ref{assum:UnderHerdFInal}. Hence, the reaching time
$t_A^v$ 
is finite. By {\em(A)} and {\em(C)} of Theorem \ref{prop:OptimumNoDelay} we have that $v$ is bang-bang with a swithcing time   $ t_\star^v\ge t_A^v$, and thus we have $v=v_M$ a.e.\ in $(0,t_A^v)$. On the other hand, it is clear that $t_A^v$ depends only on the value of $v$ before $t_A^v$, which, in our case is the constant $v_M$. In other words, $t_A^v=t_A^{v_M}$. 

By  forward invariance, once $A$ has been reached, that is after $t_A^{v_M}$, by the principle of optimality 
 $v(\cdot+t_A^{v_M})$ is optimal also for the unconstrained problem $\cP^{T-t_A^{v_M},s_1,i_1}_{\lambda_v,\lambda_i,1}$ with initial condition $(s_1,i_1)=(s(t_A^{v_M}),i(t_A^{v_M}))$. Since, 
% \begin{eqnarray}
% \bar{t}(s_1,i_1)&=&\sup\Big\{\inf_{t\ge0}\{s^{v,s_1,i_1}(t)\le\frac{\gamma}{\beta}\}\ :\ v\in\cV\Big\}\nonumber\\
% &=&
% \sup\Big\{
% \inf_{\tau\ge t_A^{v_M}}\{s^{v,s_0,i_0}(\tau)\le\frac{\gamma}{\beta}\}\ :\ v\in\cV,\,
% v_{|[0,t_A^{v_M}]}=v_M\Big\}- t_A^{v_M}\label{ttt}\\
% &\le& \sup_{v\in\cV}\inf_{\tau\ge 0}\{s^{v,s_0,i_0}(\tau)\le\frac{\gamma}{\beta}\}- t_A^{v_M}
% =\bar{t}(s_0,i_0)- t_A^{v_M}<T- t_A^{v_M},\nonumber
% \end{eqnarray}
by assumption,  $\cP^{T-t_A^{v_M},s_1,i_1}_{\lambda_v,\lambda_i,1}$ admits a unique solution then   
$v$ is uniquely determined also in $(t_A^v, T)$, and the proof is concluded.
\end{proof}

\begin{remark} Problem  $\cP^{T-t_A^{v_M},s_1,i_1}_{\lambda_v,\lambda_i,1}$ satisfies Assumption~\ref{assum:UnderHerdFInal}, because $T-t_A^{v_M}>\bar{t}(s_1,i_1)$. Indeed,   
\begin{eqnarray*}
\bar{t}(s_1,i_1)&=&\sup\Big\{\inf_{t\ge0}\{s^{v,s_1,i_1}(t)\le\frac{\gamma}{\beta}\}\ \vert\ v\in\cV\Big\}\nonumber\\
&=&
\sup\Big\{
\inf_{\tau\ge t_A^{v_M}}\{s^{v,s_0,i_0}(\tau)\le\frac{\gamma}{\beta}\}\ \vert\ v\in\cV,\,
v_{|[0,t_A^{v_M}]}=v_M\Big\}- t_A^{v_M}\label{ttt}\\
&\le& \sup_{v\in\cV}\inf_{\tau\ge 0}\{s^{v,s_0,i_0}(\tau)\le\frac{\gamma}{\beta}\}- t_A^{v_M}
=\bar{t}(s_0,i_0)- t_A^{v_M}<T- t_A^{v_M}\nonumber
\end{eqnarray*}
 Hence, Proposition \ref{prop:LambdaSmallEnough} applies and 
 the assumption of the previous theorem (i.e., the uniqueness of the solution to  $\cP^{T-t_A^{v_M},s_1,i_1}_{\lambda_v,\lambda_i,1}$) is satisfied for non-trivial values of $\lambda_i$ and $\lambda_v$.
\end{remark}

\begin{remark}
The uniqueness for the unconstrained case has been conjectured in Remark 2 of~\cite{Behncke}.
If this conjecture was true,  then we would have fully proved the uniqueness also in the constrained case.  
\end{remark}

\subsection{Numerical Simulations} \label{sec:NumericalSim} We conclude this section by presenting some numerical simulations that illustrate the previous results. 
They have been obtained by using the open-source optimal control toolbox \textsc{Bocop} (\cite{BocopExamples,bonnans2012bocop}) with final time $T=400$ and epidemic parameters $\beta=0.18$, $\gamma=0.07$, $s_0=0.7$, $i_0=0.001$, $i_M=0.005=5i_0$, $v_M=0.01$. The chosen numerical method is the recommended Midpoint (implicit, 1-stage, order 1) with 1200 time steps.  

The figures below show the optimal control and the state of infections for objective functionals corresponding to different choices of the relative costs of the vaccination $v$ (modulated by $\lambda_v$) and of  the population $i$  of infected individuals $i$  (modulated by $\lambda_i$).

% In particular, Figure \ref{l1=0} deals with the case $\lambda_i=0$ in which the cost depends only on the vaccination rate $v$.  According to the feedback optimal strategy shown in Figure \ref{optl10_fig}, we see that, in this case, the peak of infections (corresponding to herd immunity) occurs long time after having stopped to control. Another simulation, made with $\lambda_i=0.1\lambda_v>0$,  produced exactly the same results. The latter is in agreement with Proposition \ref{pr_l2lebl1} since  $\lambda_i<\beta\lambda_v$. A simulation with $\lambda_i=0.17\lambda_v$ (still smaller, but closer, to $\beta\lambda_v$) is shown in Figure \ref{l2=017l1_fig}. In this case a structure like that of  Figure \ref{optl2_betal1_fig}
% is more evident.

% Figure 
% \ref{l1=l2=1} shows the simulation results in the case $\lambda_i=\lambda_v=1$ in which the cost of the infected population  forces the state $i$ to stay strictly below the ICU threshold $i_M=0.005$. In this case it is convenient to continue the vaccination policy with the maximum effort even after that herd immunity has been reached. The corresponding feedback control policy in shown in Figure \ref{optl2big_fig}.

\begin{figure}[h]
\begin{minipage}{0.32\textwidth}
\begin{center}
\includegraphics[width=0.99\textwidth]
{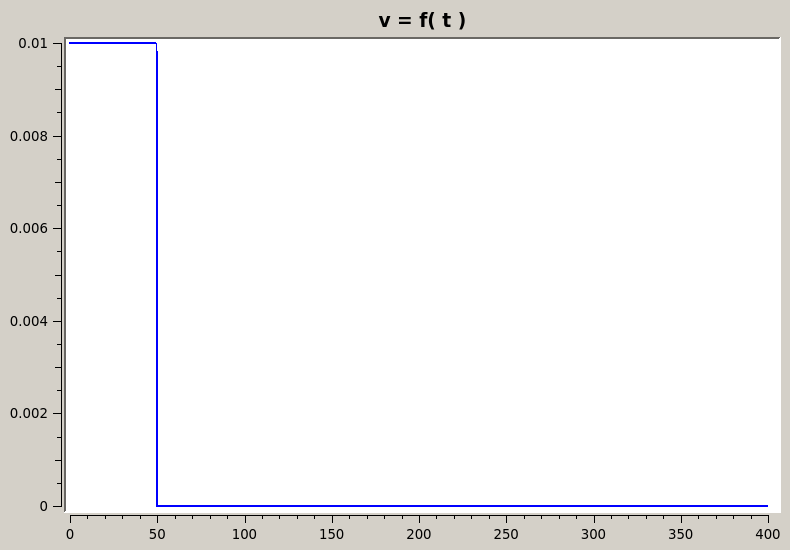}\\
\includegraphics[width=0.99
\textwidth]
{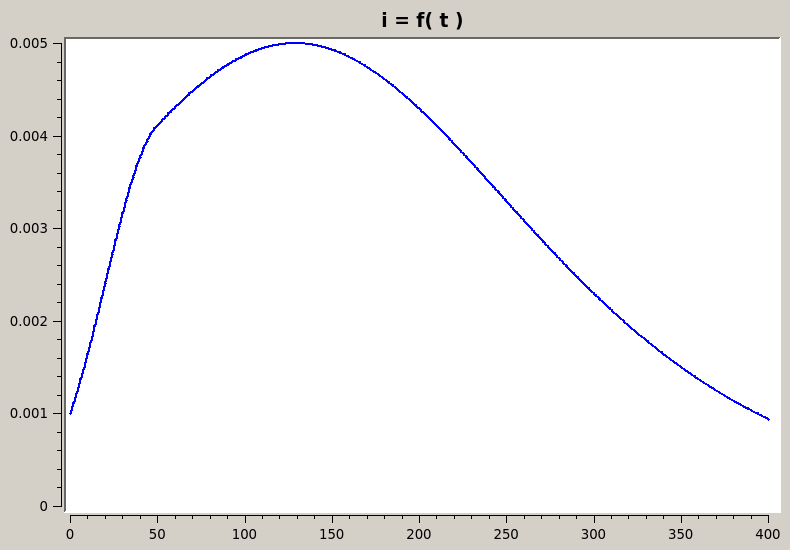}\\ 
\includegraphics[width=0.99
\textwidth]{vaccination_viability_curves_realscen_opt_star}\\
a) $\lambda_i=0$ and $\lambda_i=0.1\lambda_v<\beta\lambda_v$
\end{center}
\end{minipage}
\begin{minipage}{0.32\textwidth}
\begin{center}
\includegraphics[width=0.99\textwidth]
{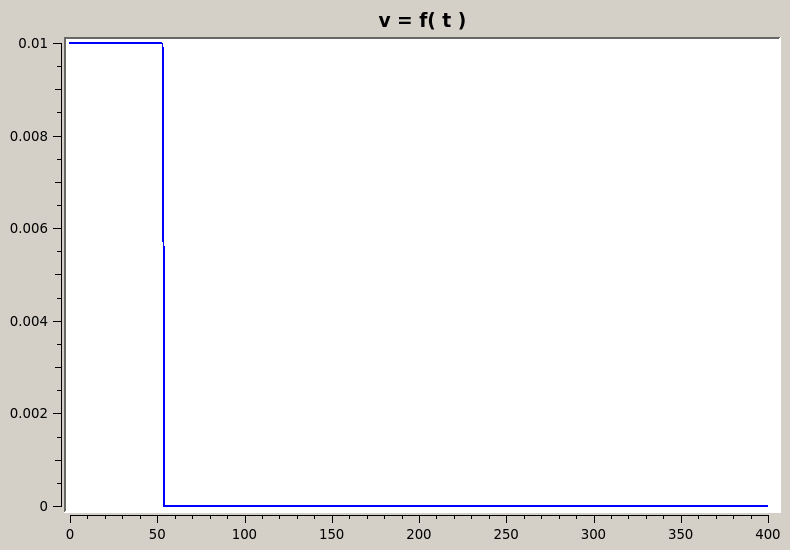}\\
\includegraphics[width=0.99\textwidth]
{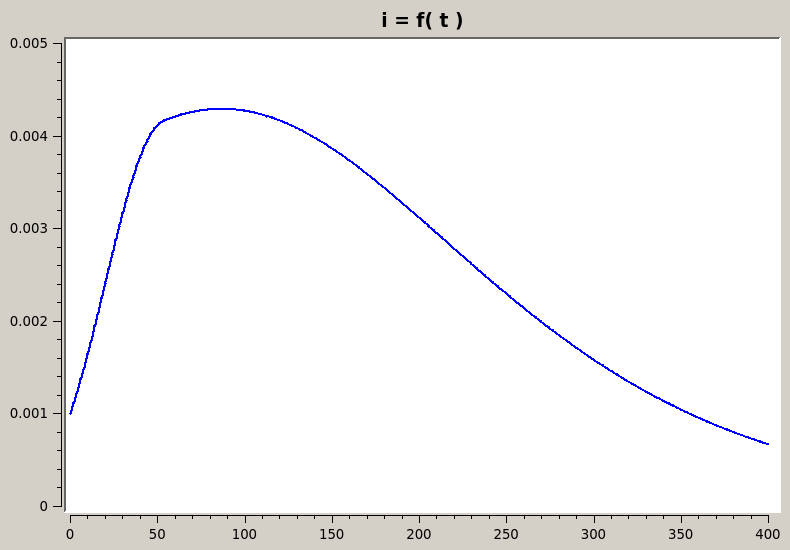}\\
\includegraphics[width=0.99
\textwidth]{vaccination_viability_curves_l2_betal1_opt_star}\\
b) $\lambda_i=0.17\lambda_v<\beta\lambda_v$
\end{center}
\end{minipage}
\begin{minipage}{0.32\textwidth}
\begin{center}
\includegraphics[width=0.99\textwidth]
{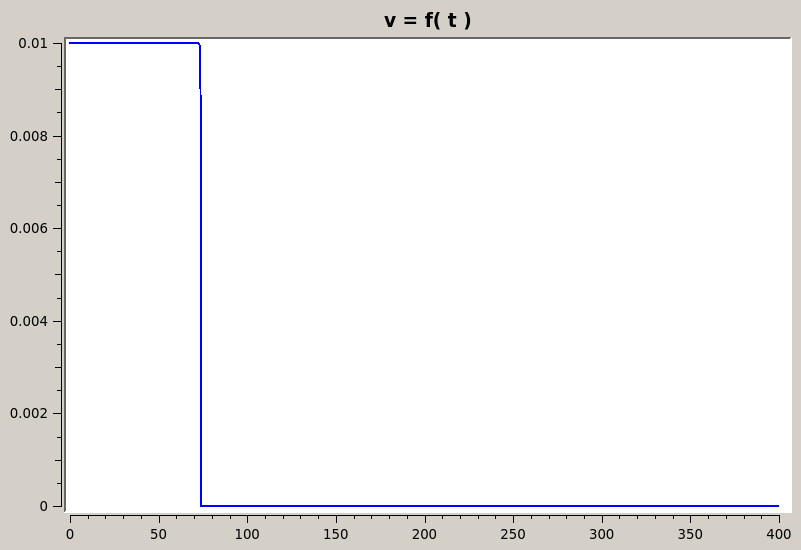}\\
\includegraphics[width=0.99\textwidth]
{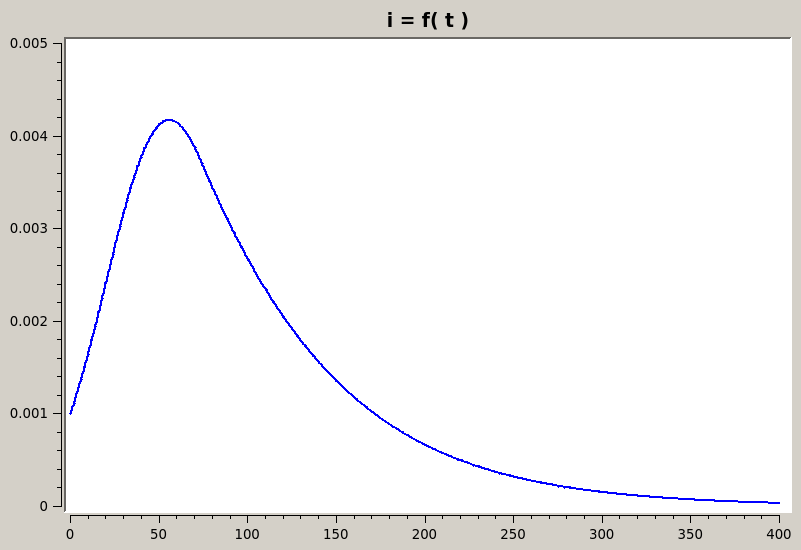}\\
\includegraphics[width=0.99
\textwidth]{vaccination_viability_curves_l2big_opt_star}\\
c) $\lambda_i=\lambda_v>\beta\lambda_v$
\end{center}
\end{minipage}
\caption{The three columns a), b) and c) show optimal control, state of infections and feedback optimal strategy for various choices of the specific costs $\lambda_v$ and $\lambda_i$.}
\label{fig_simul}
\end{figure}

In particular, Figure  \ref{fig_simul} a) deals with the case $\lambda_i=0$ in which the cost depends only on the vaccination rate $v$.  According to the feedback optimal strategy shown in Figure \ref{optl10_fig}, we see that, in this case, the peak of infections (corresponding to herd immunity) occurs long time after having stopped to control. Another simulation, made with $\lambda_i=0.1\lambda_v>0$,  produced exactly the same results, accordingly to Proposition~\ref{prop:LambdaSmallEnough}; the corresponding figures are not reported since identical to the ones in Figure~\ref{optl10_fig} a). The latter is also in agreement with Proposition \ref{pr_l2lebl1} since  $\lambda_i<\beta\lambda_v$. A simulation with $\lambda_i=0.17\lambda_v>0$ (still smaller, but closer, to $\beta\lambda_v$) is shown in Figure \ref{fig_simul} b). In this case the state-space behaviour of such strategy is qualitatively as in  Figure \ref{optl2_betal1_fig}.

Figure 
\ref{fig_simul} c) shows the simulation results in the case $\lambda_i=\lambda_v>0$ in which the cost of the infected population  forces the state $i$ to stay strictly below the ICU threshold $i_M=0.005$. In this case it is convenient to continue the vaccination policy with the maximum effort even after that herd immunity has been reached. The state-space  behaviour of the solution corresponding to this control policy is shown in Figure \ref{optl2big_fig}.

\section{The  infinite horizon problem} \label{sec:InfiniteHorizon}

In this section we consider   problem~\eqref{eq:OptimalControlProblemNoDelay} with $T=\infty$, that is on the infinite time horizon $[0,\infty)$.
% \begin{equation}\label{eq:INFINITEHorizonProbem}
% \begin{aligned}
% \mbox{minimize }J(v,s,i)&=\int_0^{\infty}\lambda_v v(t)+\lambda_i i(t)  \,dt\;\\
% &(s(t),i(t))'=f(s(t),i(t),v(t))\quad\forall\;t\geq 0,\\
% &i(t)-i_M\leq 0,\\forall\;t\geq 0,\\
% &(s(0),i(0))\in \R^2,\\
% &v(t)\in V:=[0, v_M].
% \end{aligned}
% \end{equation}
% where, in this case and throughout this section,  $\cV=L^\infty([0,\infty),V)$. 

From now on,  the initial condition $(s_0,i_0)$ is taken in the viable set $B$. Moreover, we recall that, given any $\lambda_v\geq 0$ and $\lambda_i\geq 0$ with $(\lambda_v,\lambda_i)\ne(0,0)$, and for any $T\in [0,\infty)$, we denote by $\cP^T_{\lambda_v,\lambda_i,i_M}$ the optimal control problem in~\eqref{eq:OptimalControlProblemNoDelay}, while the infinite horizon problem is denoted by $\cP^\infty_{\lambda_v,\lambda_i,i_M}$.

\subsection{Reduction to a finite horizon}\label{RFHP}
% Assume $v\in \cV$ to be an optimal control for problem $\mathcal{P}^\infty_{{\lambda_v},{\lambda_i}}$. By ~\emph{(6)} of Theorem~\ref{th_euri} we have
% $\lim_{t\to \infty}s^v(t)=:s_\infty^v<\frac{\gamma}{\beta}$, moreover the time to reach the value $s^v=\frac{\gamma}{\beta}$ with the control $v$, denoted by  $t^v_{\frac{\gamma}{\beta}}$, is finite, recall Item~\emph{(7)} of Theorem~\ref{th_euri} {\color{black} che sia finito \'e ovvio per (6). Invece (7) dice che \'e indipendente dal controllo. Ma serve? In ogni caso va detto meglio.}.
If $\lambda_i=0$, 
%after time $t^v_{\frac{\gamma}{\beta}}$ the optimal control is identically $0$, because $i$ is decreasing and the ICU constraint is trivially satisfied, and any other choice of $v$ would lead to a bigger cost. This 
we are allowed to use the preceding results for the finite-horizon problem on $[0,T]$ (for $T$ large enough)  to deduce necessary (Pontryagin) optimality conditions  for the infinite-horizon case. 

% We formalize the argument in the case $\lambda_i=0$, in the following statement. {\color{black} Risulta troppo scollegato dalla frase precedente. Riscrivere meglio.}

\begin{theorem}[Reduction to a finite horizon, case $\lambda_i=0$]\label{thm:ReductionFiniteTime}
Suppose $\lambda_i=0$ and let $v\in \cV$ be an optimal control for problem $\mathcal{P}^\infty_{{\lambda_v},0,i_M}$. 
 For every $T>\overline t(s_0,i_0)$ (defined in~\eqref{eq:DefinitionT0}), the restriction $v_T:=v_{|_{[0,T]}}$ is an optimal control for the finite horizon problem $\mathcal{P}^T_{{\lambda_v},0,i_M}$. 
\end{theorem}

\begin{proof}  By contradiction, assume that $v_T\in L^\infty([0,T],V)$ is not optimal for $\cP^T_{\lambda_v,0,i_M}$. Since, on the other hand, an optimal control exists, then there exists $w\in L^\infty([0,T],V)$ such that
$$
J^T(w,s^w,i^w)<J^T(v_T,s^{v_T},i^{v_T}).
$$ 
Since by assumption we have $T>\bar{t}(s_0,i_0)$, by~definition of $\bar{t}(s_0,i_0)$ (see~\eqref{eq:DefinitionT0}) it holds that $s^w(T)<\frac{\gamma}{\beta}$.
Let us denote by $w_0$ the extension of $w$ to the interval $[0,\infty)$ by defining it equal to $0$ in $(T,\infty)$. It is clear, from uniqueness of the solution  to the Cauchy problem for the system of  state equations, that $s^{w_0}(T)=s^{w}(T)<\frac{\gamma}{\beta}$, and thus, by the performed viability analysis, that $w_0$ is an admissible control. This implies
\begin{equation*}
J^\infty(w_0,s^{w_0},i^{w_0})=\int_0^\infty\lambda_v w_0(t)\,dt =
\int_0^T\lambda_v w(t)\,dt<\int_0^T\lambda_v v(t)\;dt \le J^\infty(v,s^v,i^v),
\end{equation*}
which contradicts the optimality of $v$ for $\cP^\infty_{\lambda_v,0,i_M}$. 
\end{proof}

Summarizing, when the cost does not depend on the number of infected individuals, the fact that the herd-immunity threshold is reached in finite time allows us to reduce the infinite-horizon problem to a finite-horizon one.  We collect this argument in the subsequent statement.

\begin{cor}\label{cor:InfiniteTrivial}
Suppose $\lambda_i=0$. Then, for any $(s_0,i_0)\in B$,  problem $\cP^\infty_{\lambda_v,0,i_M}$ admits a unique  optimal control $v\in \cV$ (with finite cost), of the form 
\[
v(t)=\begin{cases}
v_M\;\;\;&\text{if}\;t\in (0, \tsw),\\
0\;\;\;&\text{if}\;t\in (\tsw,\infty).
\end{cases}
\]
for a $\tsw\in [0, \overline t(s_0,i_0)]$.
In particular, $v$ satisfies the feedback representation as in~\eqref{eq:feedbackRepresentation}. 
\end{cor}

\begin{proof}
% Let us consider an optimal control $v\in \cV$ for the problem~\eqref{eq:INFINITEHorizonProbem}. By Theorem~\ref{thm:ReductionFiniteTime}, for any $T>\overline t$ we have that $v_{|_{[0,T]}}$ is optimal for the finite-horizon problem~\eqref{eq:OptimalControlProblemNoDelay} with $T=T$. Recalling the discussion provided in Proposition~\ref{prop:OptimumNoDelay} and Remark~\ref{rem:ParticularAndUniqueness}, we thus have that, in the interval $[0,T]$,  $v$ coincides with the \emph{unique} optimal control $\hv_T$ of the finite-time horizon problem~\eqref{eq:OptimalControlProblemNoDelay} with $T=T$, which is of the bang-bang form in~\eqref{eq:BangBangCOntrol}. We conclude noting again that, since $T>\overline t$ we have that $v(t)=0$ for all $(T,\infty)$ is a feasible (and trivially optimal) choice. The feedback representation again follows from the discussion provided in Proposition~\ref{prop:OptimumNoDelay}.

Let us consider an optimal control $v\in \cV$ for problem~$\cP^\infty_{\lambda_v,0,i_M}$. By Theorem~\ref{thm:ReductionFiniteTime}, for any $T>\overline t(s_0,i_0)$ we have that $v_{|_{[0,T]}}$ coincides with  the unique (by \emph{(B)} of Theorem \ref{prop:OptimumNoDelay})  optimal control for the finite-horizon problem $\cP^T_{\lambda_v,0,i_M}$, which is of the bang-bang form in~\eqref{eq:BangBangCOntrol} and admits the feedback representation~\eqref{eq:feedbackRepresentation}. We conclude by noting  that, since $T>\overline t(s_0,i_0)$ we have that $v(t)=0$ for all $(T,\infty)$ is a feasible and trivially optimal choice. 
\end{proof}

When the cost depends also on  $i$ (i.e., when $\lambda_i>0$), which, differently from the control, vanishes only as $t\to\infty$, the aforementioned argument cannot be applied. Nevertheless, in what follows we show that a reduction to finite horizon problems can be  obtained by a $\Gamma$-convergence argument. For a throughout and formal introduction to $\Gamma$-convergence, we refer to~\cite{DMas,Bra02book}.

In what follows, we denote by $I$ the infinite-horizon interval $[0,\infty)$.

\begin{theorem}\label{thm:GCTconv}
For every increasing sequence of positive numbers $T_n\to\infty$, the sequence of problems $\mathcal{P}^{T_n}_{{\lambda_v},{\lambda_i},i_M}$ $\Gamma$-converges to the limit problem 
$\mathcal{P}^\infty_{{\lambda_v},{\lambda_i},i_M}$, in the following sense:
\begin{enumerate}[leftmargin=*]
\item \emph{(liminf inequality).} For every sequence $v_n\in \cV$ such that $v_n=0$ a.e. in $(T_n,\infty)$,  $i^{v_n}\le i_M$, and 
$v_n\weaks v$ in $L^\infty(I,V)$ we have $i^v\le i_M$ and %$s_\infty^u\le\frac{\gamma}{\beta}$ and
$$
\liminf_{n\to\infty}J^{T_n}(v_n,s^{v_n},i^{v_n})\ge J^\infty(v,s^v,i^v);
$$ 
\item \emph{(recovery sequence).} For every $v\in \cV$ such that $J^\infty(v,s^v,i^v)<\infty$ and $i^v\le i_M$, there exists $v_n\in \cV$ such that $i^{v_n}\le i_M$, $v_n=0$ a.e. in $(T_n,\infty)$,
$v_n\weaks v$ in $L^\infty(I,V)$ and
$$
\lim_{n\to\infty}J^{T_n}(v_n,s^{v_n},i^{v_n})= J^\infty(v,s^v,i^v).
$$ 
\end{enumerate}
\end{theorem}
As a preliminary step to  the proof of the theorem, we state an useful lemma on  convergence of solutions to (controlled) differential equations, whose proof is the same of~\cite[Lemma 7.5]{FreddiGoreac23}, because the arguments used there work also for our (albeit different) state equations.

\begin{lemma}\label{lem:conv}
If $v_n\weaks v$ in $L^\infty(I;V)$, then $s^{v_n}\to s^v$ and $i^{v_n}\to i^v$ uniformly on every bounded subinterval $J\subset I$.
%$(s^{u_n},i^{u_n})\weaks (s^u,i^u)$ in $W^{1,\infty}(I;\R^2)$.
\end{lemma}

\begin{proof}[Proof of Theorem \ref{thm:GCTconv}]  Let us prove statement \emph{(1)} first. By the previous lemma, the assumption $v_n\weaks v$ in $L^\infty(I;V)$ implies 
that $i^{v_n}$ converges to $i^{v}$ uniformly on the bounded subintervals of $I$, and this in particular implies $i^v(t)\le i_M$ for every $t\in I$ (since we recall from Theorem~\ref{th_euri} that $\lim_{t\to \infty}i^v(t)=0$). 
Let us now remark that 
\begin{equation}\label{eq:ConvIvnTechnical}
\liminf_{n\to\infty}\int_0^{T_n}i^{v_n}=\liminf_{n\to\infty}\int_0^{\infty}i^{v_n}1_{[0,T_n]}\,dt\ge \int_0^{\infty}i^{v}\,dt
\end{equation}  since the sequence $i^{v_n}1_{[0,T_n]}$ converges to $i^{v}$ pointwise on $(0,\infty)$ and by using Fatou's lemma (see, for instance, \cite[Theorem 11.31]{RudinPMA}).
Then, by using~\eqref{eq:ConvIvnTechnical}, the fact that $v_n=0$ a.e. in $(T_n,\infty)$ and by the fact that the functional $v\mapsto \int_0^\infty v\;dt$ is weakly* lower semicontinuous in $L^\infty(I;V)$ (see~\cite[Theorem 5.1]{FreddiGoreac23}),  we have
\begin{equation}\label{liJ}\begin{aligned}
\liminf_{n\to\infty}J^{T_n}(v_n,s^{v_n},i^{v_n})=&\liminf_{n\to\infty}\int_0^{T_n}\lambda_v v_n+\lambda_i i^{v_n}\,dt=\lambda_v \liminf_{n\to\infty}\int_0^\infty v_n\,dt+\lambda_i\liminf_{n\to\infty}\int_0^{T_n}  i^{v_n}\,dt \\
\ge&\int_0^\infty\lambda_v v+\lambda_i i^{v}\,dt= J^\infty(v,s^v,i^v).
\end{aligned}
\end{equation}
concluding the first part of the proof.\\
  It remains to prove Item~\emph{(2)}, i.e,. the existence of a recovery sequence. 
Since $J^\infty(v,s^v,i^v)<\infty$,  the control $v$ belongs to $L^1(I)$ and, moreover, $\lim_{t\to \infty}s^v(t)<\frac{\gamma}{\beta}$, recalling~Theorem~\ref{th_euri}.
Let us consider, for any $n\in \N$,
$$
v_n(t)=\begin{cases}
v(t)&\mbox{ if }t\in[0,T_n),\\
0&\mbox{ if }t\in(T_n,\infty).
\end{cases}
$$
Since $T_n\to\infty$, for every $n$ large enough we have $s^{v_n}(T_n)=s^v(T_n)<\frac{\gamma}{\beta}$, and this in particular implies that $i^{v_n}(t)\le i_M$ for every $t\in I$. Indeed, in $(0,T_n]$, 
we have $i^{v_n}=i^v\le i_M$, while in $(T_n,\infty)$ the constraint holds since $i^{v_n}(\cdot)$ is decreasing.

Moreover, $v_n\weaks v$ in $L^\infty(I,V)$ (by Lebesgue's dominated convergence theorem) and by monotone convergence we have
\begin{equation}\begin{aligned}
\lim_{n\to\infty}J^{T_n}(v_n,s^{v_n},i^{v_n})&=\lim_{n\to\infty}\int_0^{T_n}\lambda_v v_n+\lambda_i i^{v_n}\,dt=\lambda_v\lim_{n\to\infty}\int_0^{T_n} v_n\;dt+\lambda_i\lim_{n\to \infty} \int_0^{T_n} i^{v_n}\,dt\\&=\lambda_v\int_0^{\infty}v\;dt +\lambda_i\int_0^\infty i^v\;dt=J^\infty(v,s^v,i^v),
\end{aligned}
\end{equation} 
concluding the proof.
\end{proof}

In the next statement we state important consequences of Theorem~\ref{thm:GCTconv}.

\begin{cor}\label{cor:occonv} For any $T>0$, let us denote by $v_T\in\cV$ a control obtained by taking an optimal control for $\cP^{T}_{{\lambda_v},{\lambda_i},i_M}$ and extending it by zeroes on $(T,\infty)$. Then 
\begin{enumerate}[leftmargin=*]
\item there exists an increasing sequence $T_n\to\infty$ and $v\in \cV$ such that $v_{T_n}\weaks v$;
\item if $(T_n)$ is sequence of positive numbers such that $T_n\to\infty$ and $v_{T_n}\weaks v$, then $v$ is an optimal control for problem $\mathcal{P}^{\infty}_{{\lambda_v},{\lambda_i},i_M}$.
\end{enumerate}
\end{cor}

\begin{proof}Part \emph{(1)} of the statement follows by the fact that $(v_T)$ is an equi-bounded family in $L^\infty(I;V)$ and the existence of a weak* converging sub-sequence follows by Alaoglu's theorem.

We now prove Item~\emph{(2)}. By the fact that  $(v_{T_n},s^{v_{T_n}},i^{v_{T_n}})$ are equi-bounded (recall Theorem~\ref{th_euri} and Remark~\ref{rem:Boundedness}) it can be seen that  $(v_{T_n},s^{v_{T_n}},i^{v_{T_n}})$ weakly* converges in $\cV\times W^{1,\infty}(I,\R^2)$  to $(v,s^{v},i^{v})$, and the claim follows by the variational property of $\Gamma$-convergence
(see, for instance, \cite[Corollary 7.17]{DMas}).
\end{proof}

\begin{remark}
As a consequence, problem $\mathcal{P}^{\infty}_{{\lambda_v},{\lambda_i},i_M}$ admits a solution. Moreover such solution has a finite cost: this is a consequence of Lemma~\ref{lemma:BoundIntegral}, since it implies that any feasible control $v\in L^1([0,\infty);V)$ has finite cost. 
\end{remark}

%
%{\color{black}
%What I meant to say is: 
%- in the arguments, we say that as soon as one reaches the $\mathcal{A}$ zone, $0$ is optimal. This is true provided that this control $0$ be admissible (in the sense that it realizes $s(T)\leq \frac{\gamma}{\beta}-\varepsilon$). Then, there is a play between small $\varepsilon$ and sufficient $T_\varepsilon$ to achieve this and the quantities may not be uniform in $(s_0,i_0)$. To make it simple: the time to reach $\frac{\gamma}{\beta}$ when starting on the boundary of $\mathcal{A}$ depends on $i_0$ and is infinite when $i_0\rightarrow 0+$ (the limit case...). As such, the time $T$ to get $s(T)\leq \frac{\gamma}{\beta}-\varepsilon$ depends on how small $i_0$ was...
%- considering $s(T)\leq \frac{\gamma}{\beta}-\varepsilon$ describes a closed target set, and this is the framework considered in the cited references.
%But, again, probably it is safe to take this small paragraph out altogether.
%}
%  

\subsection{Characterization of an optimal control}
Somehow collecting all the previous results, we now characterize an  optimal control for the infinite-horizon control problem.

\begin{theorem}\label{thm:CarInfiniteHor}
   Consider any $(s_0,i_0)\in B$. The problem $\cP^\infty_{\lambda_v,\lambda_i,i_M}$
admits an optimal control $\hv\in \cV$ in the bang-bang form, i.e., there exists $\overline{t}_\star\in [0,\infty]$ such that 
\begin{equation}\label{vlim}
\hv(t)=\begin{cases}
v_M,\;\;\;&\text{if } t\in [0,\overline{t}_\star),\\
0,\;\;\;&\text{if }t\in (\overline{t}_\star,\infty).
\end{cases}
\end{equation}
Moreover, 
\begin{enumerate}[leftmargin=*]
\item $\lambda_i=0$ implies $\overline{t}_\star\leq \bar t(s_0,i_0)$, $v$ is the unique optimal control and admits  the feedback representation~\eqref{eq:feedbackRepresentation},
\item $\lambda_v=0$ if and only if 
$\overline{t}_\star=\infty$.
\end{enumerate}
\end{theorem}

\begin{proof} 
By  Corollary~\ref{cor:occonv},  there exist an optimal control $\hv$  for $\cP^\infty_{\lambda_v,\lambda_i,i_M}$ and an increasing  sequence $T_n\to \infty$ such that 
$v^{T_n}\weaks \hv$, where $v^{T_n}$ is 
the extension by zeroes of an optimal control for $\cP^{T_n}_{\lambda_v,\lambda_i,i_M}$.
 By Theorem~\ref{prop:OptimumNoDelay}, for any $n$,  there exists $\tsw^n\in [0,T_n)$ such that  
 $v^{T_n}=v_M 1_{[0,\tsw^n)}$. 
% \[
% v^{T_n}(t)=\begin{cases}
% v_M,\;\;\;&\text{if } t\in [0, t^n_s),\\
% 0,\;\;\;&\text{if } t\in  (t^n_s,\infty).
% \end{cases}
% \]
% By convention {\color{black} why, "by convention"?}, for any $T>0$, we extend the control $v^T$ in $(T,\infty)$ by defining it equal to $0$, in $(T,\infty)$. 
  This implies that 
 there exist   $\overline{t}_\star\in [0,\infty]$ and a subsequence $(n_k)$ such that $\tsw^{n_k}\to \overline{t}_\star$. By Lebesgue dominated convergence theorem, for every $\varphi\in L^1(0,\infty)$ we have
 $$
\lim_{k\to\infty}\int_0^\infty v^{T_{n_k}}\varphi\,dt=\lim_{k\to\infty}\int_0^{t^{n_k}_s} v_M\varphi\,dt=
\lim_{k\to\infty}\int_0^{\overline{t}_\star}v_M\varphi\,dt.
 $$
Thus, $v^{T_{n_k}}\weaks v_M1_{[0,{\overline{t}_\star})}$ and the claim \eqref{vlim} follows 
by uniqueness of the limit. 

Assertion \emph{(1)} of the moreover part of the statement is provided by Corollary~\ref{cor:InfiniteTrivial}. Let us now prove \emph{(2)}.
 If $\lambda_v>0$, the case $\overline{t}_\star=\infty$ is excluded, since otherwise the optimal cost would be infinite, in contradiction with the fact that any feasible control $v\in L^1(I;V)$  has finite cost, by Lemma~\ref{lemma:BoundIntegral}. If $\lambda_v=0$, by \emph{(A)} of Theorem~\ref{prop:OptimumNoDelay} we have $t_\star^n=T_n$ for all $n\in \N$, and thus $\overline{t}_\star=\lim_{n\to \infty}T_n=\infty$, concluding the proof.
\end{proof}

% {\color{blue}
% \subsection{The limit Hamiltonian}

% Let $(s_0,i_0)\in B$.   By Corollary \ref{cor:occonv}, there exist an optimal control  $v\in\cV$  for problem $\mathcal{P}^{\infty}_{{\lambda_v},{\lambda_i},i_M}$ and 
%   an increasing sequence $T_n\to\infty$  such that the solution $v^{n}$ of problem $\mathcal{P}^{T_n}_{{\lambda_v},{\lambda_i},i_M}$ weakly* converges to $v$. 
% By the same notation $v^n$ we denote the extension by zeroes of $v^n$ to $(T_n,\infty)$. Setting $I=(0,\infty)$, let us denote by $s^n\in W^{1,\infty}(I)$ and  $i^n\in  C^1(I)$  the states corresponding to the control $v^n$ and the initial conditions $s_0$  and $i_0$. 

%   By 
%   $p_0^n$,  $p_s^n$, $p_i^n$, $d\mu^n$ the states, co-states and multiplicators associated to $v^n$ on the interval $[0,T_n]$. They satisfy conditions $(A)-(F)$ of Lemma \ref{lemma:PontryaginNODelay} and the subsequent Lemma \ref{Lemma:Technical}.  

% {\color{blue}Open problem: is the sequence $d\mu^n$ (suitably extended) of bounded variation (i.e., bounded in $M(\R)$)? In other words, knowing that $\mu^n=a_n\delta_{t_n}$, can we say that the sequence $(a_n)$ is bounded? 
% Numerical simulations say "yes".}  

% Since $p_0^n$ and $\mu^n$ are bounded sequences, then, up to a subsequence (not relabeled) we have $p_0^n\to p_0\in\{0,1\}$ and $\mu_n\weaks\mu$ in $M$  
%   }

\section{Conclusions}

In this paper  we have analyzed the problem of optimal vaccination for an SIR model that includes a constraint on the ICU capacity. Our results can be summarized as follows.
\begin{itemize}[leftmargin=*]
\item After a preliminary viability analysis, we used Pontryagin's  necessary conditions to  prove that the optimal controls have a bang-bang structure {\color{black} with at most one switch.  The resulting optimal strategy consists in starting the vaccination campaign at the maximal rate as soon as possible and concentrate the intervention in a unique time interval whose length depends on the cost coefficients $\lambda_i$ and $\lambda_v$.}
    \item We identified the conditions under which vaccination should be implemented or halted. In particular, we showed that when the cost $\lambda_i$ associated with infection is significantly lower than the cost of vaccination {\color{black}(i.e., $\lambda_i\le\beta\lambda_v$, see Proposition \ref{pr_l2lebl1}),} the optimal control strategy primarily focuses on minimizing the latter, resulting in a vaccination policy that stops before the epidemic peak. 
    If, on the contrary, the two competing costs are comparable, the optimal strategy results in a prolonged vaccination period. {\color{black}
Another important remark is that when $\lambda_i$ is small enough depending on  $\lambda_v$, any optimal control   coincides with the one corresponding to the case $\lambda_i=0$.  This, in particular, proves the uniqueness of the optimal control whenever the cost of the vaccination program is ``prevalent'' with respect to the cost of treatment of infected individuals. }
    \item Numerical simulations were conducted to illustrate the theoretical results, confirming the effectiveness of the optimal control strategies derived in our analysis.
\end{itemize}
These results contributed to the understanding of optimal vaccination policies in the presence of ICU constraints and thus provide valuable insights for the design of public health interventions during epidemic outbreaks.

\noindent
{\bf Acknowledgements.}  
 M.D.R. and L.F. are members of GNAMPA--INdAM.

\

\bibliography{bib_SIR} 
\bibliographystyle{abbrv}

\end{document}